\documentclass[fleqn, final]{zzz}

\usepackage{mathrsfs}
\usepackage{color}
\usepackage{tikz}
\usepackage{float}
\usepackage{graphicx}
\usepackage{rotating}
\usepackage[pdftex]{hyperref}
\usepackage{enumerate}
\usepackage{soul}

\hypersetup{
	colorlinks = true,
	urlcolor = blue,
	linkcolor = red!70!black,
	citecolor = green!70!black
}

\definecolor{Mycolor2}{HTML}{e85d04}
\sethlcolor{black}
\newcommand\moro[1]{{\textcolor{black}{#1}}} 

\newcommand\boro[1]{{\textcolor{black}{#1}}}

\newcommand{\II}[1]{\overset{\raisebox{0.5ex}{$#1$}}{\raisebox{-0.55cm}{\resizebox{0.3cm}{!}{\scalebox{0.4}[1.0]{\mbox{$\mathbb{I}$}}}}}}

\newcommand{\rphis}[2]{{}_{#1\vphantom{#2}}\phi_{#2\vphantom{#1}}}

\newcommand{\rphisx}[4]{\rphis{#1}{#2}\left( \begin{array}{c} #3 \end{array};q,#4\right)}

\newcommand{\Whyp}[5]{\,\mbox{}_{#1}W_{#2}\!\left({#3};{#4};{#5}\right)}
\newcommand{\qhyp}[5]{\,\mbox{}_{#1}\phi_{#2}\!\left(
\genfrac{}{}{0pt}{}{#3}{#4};#5\right)}
\newcommand{\qphyp}[6]{\,\sideset{_{#1}^{\phantom{\mid}}}{_{#2}^{#3}}{\mathop{\phi}}\!\left(\genfrac{}{}{0pt}{}{#4}{#5};#6\right)}
\newcommand{\nqphyp}[3]{\,\sideset{_{#1}^{\phantom{\mid}}}{_{#2}^{#3}}{\mathop{\phi}}}
\newcommand{\nlqphyp}[3]{\,\sideset{_{#1}}{_{#2}^{#3}}{\mathop{\phi}}}

\newcommand{\hyp}[5]{\,\mbox{}_{#1}F_{#2}\!\left(
 \genfrac{}{}{0pt}{}{#3}{#4};#5\right)}

\newtheorem{thm}{Theorem}[section]
\newtheorem{cor}[thm]{Corollary}

\newtheorem{rem}[thm]{Remark}
\newtheorem{lem}[thm]{Lemma}
\newtheorem{defn}[thm]{Definition}

\makeatletter
\def\eqnarray{\stepcounter{equation}\let\@currentlabel=\theequation
\global\@eqnswtrue
\tabskip\@centering\let\\=\@eqncr
$$\halign to \displaywidth\bgroup\hfil\global\@eqcnt\z@
$\displaystyle\tabskip\z@{##}$&\global\@eqcnt\@ne
\hfil$\displaystyle{{}##{}}$\hfil
&\global\@eqcnt\tw@ $\displaystyle{##}$\hfil
\tabskip\@centering&\llap{##}\tabskip\z@\cr}

\def\endeqnarray{\@@eqncr\egroup
\global\advance\c@equation\m@ne$$\global\@ignoretrue}

\def\@yeqncr{\@ifnextchar [{\@xeqncr}{\@xeqncr[5pt]}}
\makeatother

\newcommand{\Z}{\mathbb{Z}} 
\newcommand{\N}{\mathbb{N}} 

\newcommand{\CC}{{{\mathbb C}}}
\newcommand{\CCast}{{{\mathbb C}^\ast}}
\newcommand{\CCdag}{{{\mathbb C}^\dag}}
\newcommand{\SSS}{{\mathcal S}}

\newcommand{\expe}{{\mathrm e}}

\usepackage{scalerel,url}
\let\svus_
\catcode`_=\active %
\def_{\ifmmode\expandafter\svus\expandafter\bgroup\expandafter\lowerit
\else\svus\fi}
\def\lowerit#1{\ThisStyle{\raisebox{-2\LMpt}{$\SavedStyle#1$}}\egroup}
\makeatletter
 \AtBeginDocument{
 \check@mathfonts
 \fontdimen13\textfont2=3.5pt
 \fontdimen14\textfont2=3.5pt
 \fontdimen16\textfont2=2.5pt
 \fontdimen17\textfont2=2.5pt
 }
\makeatother
\begin{document}

\renewcommand{\PaperNumber}{***}
\renewcommand{\PublicationYear}{2022}
\renewcommand{\LastPageEnding}{21}
\FirstPageHeading

\ShortArticleName{Utility of integral representations for basic hypergeometric functions
}
\ArticleName{Utility of integral representations for\\basic hypergeometric 
functions\\and orthogonal polynomials}

\Author{Howard S. Cohl
$\,^{\dag}$ and 
Roberto S. Costas-Santos$\,^{\S}$
}

\AuthorNameForHeading{H.~S.~Cohl and R.~S.~Costas-Santos}
\Address{$^\dag$ Applied and Computational 
Mathematics Division, National Institute of Standards 
and Tech\-no\-lo\-gy, Mission Viejo, CA 92694, USA
\URLaddressD{
\href{http://www.nist.gov/itl/math/msg/howard-s-cohl.cfm}
{http://www.nist.gov/itl/math/msg/howard-s-cohl.cfm}
}
} 
\EmailD{howard.cohl@nist.gov} 

\Address{$^\S$ Departamento de F\'isica y Matem\'{a}ticas,
Universidad de Alcal\'{a}, c.p. 28871, Alcal\'{a} de Henares, Spain} 
\URLaddressD{
\href{http://www.rscosan.com}
{http://www.rscosan.com}
}
\EmailD{rscosa@gmail.com} 


\ArticleDates{Received: 30 April 2021 / Accepted: 17 August 2021 / Published online: 23 June 2022}
\Abstract{
We describe the utility of integral representations for
sums of basic hypergeometric functions. 
In particular we use these to derive 
an infinite sequence 
of transformations for symmetrizations over certain 
variables which the functions possess. 
These integral representations were studied by 
Bailey, Slater, Askey, Roy, Gasper and 
Rahman and were also used to facilitate the computation
of certain outstanding problems in the theory of basic 
hypergeometric orthogonal polynomials in the $q$-Askey scheme. 
{We also generalize and give consequences and transformation formulas 
for some fundamental integrals connected to nonterminating basic 
hypergeometric series and the Askey--Wilson polynomials. We  express 
a certain integral of a ratio of infinite $q$-shifted factorials as a 
symmetric sum of two basic hypergeometric series with argument $q$. The result is then expressed as a $q$-integral.
Examples of integral representations applied to the
derivation of generating functions for Askey--Wilson are given
and as well} the
computation of a missing generating function 
for the continuous dual $q$-Hahn polynomials.
}

\Classification{33D15, 33D60}


\begin{flushright}
\begin{minipage}{70mm}
\it Dedicated to the life and mathematics\\ of Dick Askey, 1933-2019.
\end{minipage}
\end{flushright}

\section{{Introduction}}
\label{Introduction}

{
The main aim of this paper is to demonstrate
the utility of revisiting the application
of integral representations for problems
in basic hypergeometric functions and
basic hypergeometric orthogonal
polynomials in the $q$-Askey scheme. 
We accomplish this by proving a collection
of identities which arise naturally through
the utilization of this powerful method.
For a detailed history of the subject of
integral representations for basic hypergeometric
functions,
see \cite{Gasper1989} and \cite[Chapter 4]{GaspRah}.}

{
All of the results presented below are
new but some rely heavily on identities 
which have been proven elsewhere in the literature. For instance Theorem \ref{gascard} is essentially a restatement of 
\cite[(4.10.5-6)]{GaspRah}, which in turn is closely connected to \cite[{(5.2.4) and} (5.2.20)]{Slater66}. However,
our introduction of the useful $t$ parameter
in Theorem \ref{gascard} (see for instance Lemma \ref{intlem}, Corollary \ref{corsumDC} and Corollary \ref{abcd2}) is new. Also, our
utilization of the  
powerful van de Bult--Rains notation 
for basic hypergeometric series with vanishing numerator or denominator parameters 
(see 
\eqref{topzero}, \eqref{botzero} below)
in Theorem \ref{gascard} allows for a 
clear elucidation of structure which in our opinion, is not as such in previous incarnations of
this or related results in the literature.
Furthermore, even though we believe that {Theorem} \ref{gascaru} is new in its full generality, the ideas which went into it have been used many times in the literature, such as in derivations of the 
Askey--Wilson integral \eqref{AWint}, the Nassrallah--Rahman integral \eqref{NRint},
the Rahman integral \eqref{Rint}, the Askey--Roy integral \eqref{ARint} and the
Gasper integral \eqref{Gint}, as well as 
in fundamental results such as 
\cite[Exercises 4.4, 4.5]{GaspRah} which
reappear in \eqref{fourfour} and \eqref{IRAW}.
}

\subsection{Preliminaries}
\label{Preliminaries}

We adopt the following set 
notations: $\mathbb N_0:=\{0\}\cup\N=\{0, 1, 2, ...\}$, and we 
use the sets $\mathbb Z$, $\mathbb R$, $\mathbb C$ which represent 
the integers, real numbers and 
complex numbers respectively, $\CCast:=\CC\setminus\{0\}$, and 
{$\CCdag:=\{z\in\CCast: |z|<1\}$.}
We also adopt the following notation and conventions.
{Also, given a set ${\bf a}:=\{a_1,\ldots,a_A\}$, for $A\in\N$, then we define ${\bf a}_{[k]}:={\bf a}\setminus\{a_k\}$, $1\le k\le A$,
$b\,{\bf a}:=\{b\, a_1, b\, a_2, \ldots, b\, a_A\}$,
${\bf a}+b:=\{a_1+b,a_2+b,\ldots,a_A+b\}$,
where
$b,a_1,\ldots,a_A\in\mathbb C$.
}

\medskip
We assume that the empty sum 
vanishes and the empty product 
is unity.
{
We will also adopt the following symmetric sum notation.
\begin{defn}
For some function $f(a_1,\ldots,a_n;{\bf b})$,
where ${\bf b}$ is some set of parameters.
Then 
\begin{equation}
\II{a_1\boro{;a_2},\ldots,a_n}\!\!\!
f(a_1,\ldots,a_n;{\bf b})
:=f(a_1,\ldots,a_n;{\bf b})+
{\mathrm{idem}}(a_1;a_2,\ldots,a_n),
\end{equation}
where ``\,${\mathrm{idem}}(a_1;a_2,\ldots,a_n)$'' after an 
expression stands for the sum of the $n-1$ expressions 
obtained from the preceding expression by interchanging $a_1$ 
with each $a_k$, $k=2,3,\ldots,n$.
\end{defn}
}
\begin{defn} \label{def:2.1}
We adopt the following conventions for succinctly 
writing elements of {sets}. To indicate sequential positive and negative 
elements, we write
\[
\pm a:=\{a,-a\}.
\]
\noindent We also adopt an analogous notation
\[
\expe^{\pm i\theta}:=\{\expe^{i\theta},\expe^{-i\theta}\}.
\]
\noindent In the same vein, consider the numbers $f_s\in\mathbb C$ 
with $s\in{\mathcal S}\subset \N$,
with ${\mathcal S}$ finite.
Then, the notation
$\{f_s\}$
represents the set of all complex numbers $f_s$ such that 
$s\in\SSS$.
Furthermore, consider some $p\in\SSS$, then the notation
$\{f_s\}_{s\ne p}$ represents the sequence of all complex numbers
$f_s$ such that $s\in\SSS\!\setminus\!\{p\}$.
\end{defn}

Consider $q\in\CCdag$, $n\in\mathbb N_0$.
Define the sets 
$\Omega_q^n:=\{q^{-k}: k\in\mathbb N_0,~0\le k\le n-1\}$,
$\Omega_q:=\Omega_q^\infty=\{q^{-k}:k\in\mathbb N_0\}$.
In order to obtain our derived identities, we rely on properties 
of the {$q$-shifted factorial $(a;q)_n$}. {It has been pointed out 
by the referee that Askey, partly for historical reasons and partly because he preferred descriptive names to honorifics, referred to $(a;q)_n$ as a $q$-shifted factorial rather than the other common nomenclature:~$q$-Pochhammer symbol.} 
For any $n\in \mathbb N_0$, $a,{b}, q \in \mathbb C$, 
the 
$q$-shifted factorial is defined as
\begin{eqnarray}
&&\hspace{-3.3cm}\label{poch.id:1} 
(a;q)_n:=(1-a)(1-aq)\cdots(1-aq^{n-1}).
\end{eqnarray}
One may also define
\begin{eqnarray}
&&\hspace{-9cm}(a;q)_\infty:=\prod_{n=0}^\infty (1-aq^{n}), \label{infPochdefn}\\
&&\hspace{-9cm}
\label{theta}
{\vartheta(x;q):=(x,q/x;q)_\infty,}
\end{eqnarray}
{where $|q|<1$, \boro{$x\ne 0$} and \eqref{theta}
defines the modified theta function of 
nome $q$ \cite[(11.2.1)]{GaspRah}.}
\moro{Note that $\vartheta(q^n;q)=0$ for all
$n\in{\mathbb Z}$.}
Furthermore, one has the following
identities
\begin{eqnarray}
&&\hspace{-8cm}\label{sq}
(a^2;q)_\infty=(\pm a,\pm q^\frac12 a;q)_\infty,\\
&&\hspace{-8cm}\label{infone}
\frac{(a,q/a;q)_\infty}
{(qa,1/a;q)_\infty}=
{\frac{\vartheta(a;q)}{\vartheta(qa;q)}}=-a\boro{,}
\label{infneg}
\end{eqnarray}
\boro{where $a\ne 0$. Moreover}, define 
\begin{equation}
(a;q)_b:=\frac{(a;q)_\infty}{(a q^b;q)_\infty},
\label{infPochdefnb}
\end{equation}
where $a q^b\not \in \Omega_q$.
We will also use the common notational product conventions
\begin{eqnarray}
&&\hspace{-8cm}(a_1,...,a_k;q)_b:=(a_1;q)_b\cdots(a_k;q)_b.\nonumber
\end{eqnarray}

The following properties for the {$q$-shifted factorial} can be found in Koekoek et al. 
\cite[(1.8.7), (1.8.10-11), (1.8.14), (1.8.19), (1.8.21-22)]{Koekoeketal}, 
namely for appropriate values of $q,a\in\CCast$ and $n,k\in\mathbb N_0$:
\begin{eqnarray}
\label{poch.id:3} &&\hspace{-6.5cm}
(a;q^{-1})_n=(a^{-1};q)_n(-a)^nq^{-\binom{n}{2}},
\\[2mm] 
\label{poch.id:4}
&&\hspace{-6.5cm}(a;q)_{n+k}=(a;q)_k(aq^k;q)_n 
= (a;q)_n(aq^n;q)_k,\\[2mm] 
\label{poch.id:5}&&\hspace{-6.5cm} (a;q)_n
=(q^{1-n}/a;q)_n(-a)^nq^{\binom{n}{2}},\\[2mm]
\label{poch.id:6}&&\hspace{-6.5cm}(aq^{-n};q)_{k}=q^{-nk}
\frac{(q/a;q)_n}{(q^{1-k}/a;q)_n}(a;q)_k,\\[2mm]
\label{poch.id:8}&&\hspace{-6.5cm}(a^2;q^2)_n=(\pm a;q)_n,\\[2mm]
\label{poch.id:7}&&\hspace{-6.5cm}(a;q)_{2n}
=(a,aq;q^2)_n=(\pm\sqrt{a},\pm\sqrt{qa};q)_n.
\end{eqnarray}
\noindent Observe that\boro{,} by using (\ref{poch.id:4}) and (\ref{poch.id:7}), one obtains
\begin{eqnarray}
&&\hspace{-6.5cm}(aq^n;q)_n=\frac{(\pm 
\sqrt{a},\pm \sqrt{qa};q)_n}{(a;q)_n},\quad 
a\not\in\Omega_q^n. 
\label{poch.id:9}
\end{eqnarray}
\noindent {
Define the Jackson $q$-integral as in \cite[(1.11.2)]{GaspRah}
\begin{eqnarray}
&&\hspace{-2.3cm}\int_a^b f(u;q){\mathrm d}_qu=(1-q)b\sum_{n=0}^\infty q^nf(q^nb;q)-(1-q)a\sum_{n=0}^\infty q^nf(q^na;q)
\label{qint}\\
&&\hspace{0.1cm}=\frac{(1-q)ab}{a-b}
\II{a\boro{;}b}\!\!\left(1-\frac{a}{b}\right)\sum_{n=0}^\infty q^n f(q^na;q)\label{qJacsym}
\\
&&\hspace{0.1cm}{=\frac{b}{a-b}\II{a\boro{;}b}\!\!
\left(1-\frac{a}{b}\right)
\int_0^a f(u;q){\mathrm d}_qu},
\nonumber
\end{eqnarray}
where we have utilized 
\eqref{infneg} to write 
the $q$-integral as a 
symmetric sum.
}

\medskip
{
The nonterminating basic hypergeometric series, 
which we 
will often use, is defined for
{$z\in \mathbb C$},
$q\in\CCdag$, 
$s\in\mathbb N_0$, $r\in\mathbb N_0\cup\{-1\}$,
$b_j\not\in\Omega_q$, $j=1,...,s$, as
\cite[(1.10.1)]{Koekoeketal}
\begin{equation}
\qhyp{r+1}{s}{a_1,...,a_{r+1}}
{b_1,...,b_s}
{q,z}
:=\sum_{k=0}^\infty
\frac{(a_1,...,a_{r+1};q)_k}
{(q,b_1,...,b_s;q)_k}
\left((-1)^kq^{\binom k2}\right)^{s-r}
z^k.
\label{2.11}
\end{equation}
\noindent {For $s>r$, ${}_{r+1}\phi_s$ is an entire
function of $z$, for $s=r$ then 
${}_{r+1}\phi_s$ is convergent for $|z|<1$, and for $s<r$ the series
is divergent.} 
\begin{rem}
Sometimes we also use generalized hypergeometric series 
${}_{r+1}F_s$ which
is the $q\uparrow 1$ limit of basic hypergeometric series
(see for instance \cite[p.~15]{Koekoeketal}).
For their properties, see \cite[Chapter 16]{NIST:DLMF}.
\end{rem}
Note that we refer to a basic hypergeometric
series as $\ell$-balanced if
$q^\ell a_1\cdots a_{r+1}=b_1\cdots b_s$, 
and balanced 
 (Saalsch\"utzian) if $\ell=1$
{(see \cite[Definition 3.3.1]{AAR}, \cite[p.~5]{GaspRah})}.
{The referee has pointed out that
for the very important $\ell=1$ case,
the term {\it balanced} was introduced
by Askey, whereas the earlier term 
Saalsch\"utzian is due to Whipple  and was used by Bailey, but lost much of its
force after Askey's discovery in 1975 that Pfaff had Saalsch\"utz's identity 93 years
earlier in 1797.
}
A basic hypergeometric series ${}_{r+1}\phi_r$ is { well-poised} if 
the parameters satisfy the relations
\[
qa_1=b_1a_2=b_2a_3=\cdots=b_ra_{r+1}.
\]
\noindent It is {very-well poised} if in addition, 
$\{a_2,a_3\}=\pm q\sqrt{a_1}$.
}
Define the very-well poised 
basic hypergeometric series
${}_{r+1}W_r$ \cite[(2.1.11)]{GaspRah}
\begin{equation}
\label{rpWr}
{}_{r+1}W_r(b;a_4,\ldots,a_{r+1};q,z)
:=\qhyp{r+1}{r}{\pm q\sqrt{b},b,a_4,\ldots,a_{r+1}}
{\pm \sqrt{b},\frac{qb}{a_4},\ldots,\frac{qb}{a_{r+1}}}{q,z},
\end{equation} 
where $\sqrt{b},\frac{qb}{a_4},\ldots,\frac{qb}{a_{r+1}}\not\in\Omega_q$. 
When the very-well poised basic hypergeometric
series is termina\-ting, then one has
\begin{equation}\label{eq:2.13}
{}_{r+1}W_r\left(b;q^{-n},a_5,\ldots,a_{r+1};q,z\right)
=\qhyp{r+1}{r}{q^{-n},\pm q\sqrt{b}, b, a_5,\ldots,a_{r+1}}
{\pm \sqrt{b},q^{n+1}b,\frac{qb}{a_5},\ldots,\frac{qb}{a_{r+1}}}
{q,z},
\end{equation}
where $\sqrt{b},\frac{qb}{a_5},\ldots,\frac{qb}{a_{r+1}}\not\in
\Omega_q^n\cup\{0\}$.
The Askey--Wilson polynomials are intimately connected with 
the terminating very-well poised ${}_8W_7$, which is
given by
\begin{equation}
\label{VWP87}
{}_{8}W_7(b;q^{-n}\!,c,d,e,f;q,z)
=\qhyp{8}{7}{q^{-n},\pm q\sqrt{b}, b,c,d,e,f}
{\pm \sqrt{b},q^{n+1}b,\frac{qb}{c},\frac{qb}{d},\frac{qb}{e},\frac{qb}{f}}
{q,z},
\end{equation}
where $\sqrt{b},\frac{qb}{c},\frac{qb}{d},\frac{qb}{e},\frac{qb}{f}\not
\in\Omega_q^n\cup\{0\}$.

\medskip {
In the sequel, we will use the following notation 
$\nqphyp{r+1}{s}{m}$,  $m\in\mathbb Z$
(originally due to van de Bult \& Rains
\cite[p.~4]{vandeBultRains09}),  
for basic hypergeometric series with
zero parameter entries.
Consider $p\in\mathbb N_0$. Then define
\begin{equation}\label{topzero} 
\nlqphyp{r+1}{s}{-p}\!
\left(\!\begin{array}{c}a_1,\ldots,a_{r+1} \\
b_1,\ldots,b_s\end{array};q,z
\right)
:=
\rphisx{r+p+1}{s}{a_1,a_2,\ldots,a_{r+1},\overbrace{0,\ldots,0}^{p} \\ b_1,b_2,\ldots,b_s}{z},
\end{equation}
\begin{equation}\label{botzero}
\nlqphyp{r+1}{s}{\,p}\!
\left(\!\begin{array}{c}a_1,\ldots,a_{r+1} \\
b_1,\ldots,b_s\end{array};q,z
\right)
:=
\rphisx{r+1}{s+p}{a_1,a_2,\ldots,a_{r+1} \\ b_1,b_2,\ldots,b_s, \underbrace{0,\ldots,0}_p}{z},
\end{equation}
where $b_1,\ldots,b_s\not
\in\Omega_q\cup\{0\}$, and
$
\nlqphyp{r+1}{s}{0}
={}_{r+1}\phi_{s}
.$
The nonterminating basic hypergeometric series 
$\nlqphyp{r+1}{s}{m}
({\bf a};{\bf b};q,z)$, ${\bf a}:=\{a_1,\ldots,a_{r+1}\}$,
${\bf b}:=\{b_1,\ldots,b_s\}$, is well-defined for $s-r+m\ge 0$. In particular 
$\nlqphyp{r+1}{s}{m}$
is an entire function of $z$ for $s-r+m>0$, convergent for $|z|<1$ for $s-r+m=0$ and divergent if $s-r+m<0$.
Note that we will move interchangeably between the
van de Bult \& Rains notation and the alternative
notation with vanishing numerator and denominator parameters
which are used on the right-hand sides of \eqref{topzero} and \eqref{botzero}.}

\subsection{{\hspace{-2mm}{The Askey--Wilson polynomials 
and related 
fundamental integrals}}}

{Let $n\in\mathbb N_0$, $q\in\CCdag$. 
For the Askey--Wilson polynomials $p_n(x;{\bf a}|q)$, which are symmetric
in four free parameters, we will switch interchangeably
with the notation ${\bf a}:=\{a_1,a_2,a_3,a_4\}$ and 
${\mathfrak a}:=\{a,b,c,d\}$, 
${\bf a}={\mathfrak a}$,
$a,b,c,d\in\CCast$,
and  similarly for the continuous dual $q$-Hahn polynomials
which are symmetric in three free parameters.
Define
$a_{12}:=a_1a_2$,
$a_{13}:=a_1a_3$,
$a_{23}:=a_2a_3$,
$a_{123}:=a_1a_2a_3$,
$a_{1234}:=a_1a_2a_3a_4$, etc. }
\medskip

The Askey--Wilson polynomials can be defined in 
terms of the terminating basic hypergeometric series
\cite[(14.1.1)]{Koekoeketal}
\begin{equation}
p_n(x;{\mathfrak a}|q):=a^{-n}(ab,ac,ad;q)_n\qhyp43{q^{-n};
q^{n-1}abcd,a\expe^{\pm i\theta}}{ab,ac,ad}{q,q},
\quad x=\cos \theta,\label{AW}
\end{equation}
such that $ab, ac, ad\not \in \Omega_q^n$.
The Askey--Wilson polynomials are orthogonal
on $(-1,1)$
with respect to the 
weight function
\begin{eqnarray}
&&\hspace{-5.1cm}
w_q(\cos\theta;{\bf a}):=
\frac{(\expe^{\pm 2i\theta};q)_\infty}
{({\bf a}\expe^{\pm i\theta};q)_\infty}=
\label{AWweight}
\frac{(\pm\expe^{\pm i\theta},\pm q^\frac12 \expe^{\pm i\theta};q)_\infty}
{({\bf a}\expe^{\pm i\theta};q)_\infty},
\end{eqnarray}
where the second equality is due to
\eqref{sq}.
The orthogonality relation for
Askey--Wilson polynomials is
\cite[(14.1.2)]{Koekoeketal}
\begin{equation}
\int_0^\pi p_m(x;{\bf a}|q)p_n(x;{\bf a}|q)w_q(x;{\bf a})\,
{\mathrm d}\theta=h_n({\bf a};q)\delta_{m,n},
\end{equation}
where 
\begin{equation}
h_n({\bf a};q):=\frac{2\pi(q^{n-1}a_{1234};q)_n(q^{2n}
a_{1234};q)_\infty}{(q^{n+1},q^na_{12},q^na_{13}
,q^na_{14},q^na_{23},q^na_{24},q^na_{34};q)_\infty}.
\end{equation}

{
We will also rely on several important generalized
$q$-beta integrals. The first is the Askey--Wilson integral 
\cite[(6.1.1)]{GaspRah} (the integral over the full domain of the
Askey--Wilson weight \eqref{AWweight})
\begin{eqnarray}
&&\hspace{-4.3cm}\int_0^\pi\frac{(\expe^{\pm 2i\theta};q)_\infty}
{({\bf a}\expe^{\pm i\theta};q)_\infty}
{\mathrm d}\theta
=\frac{2\pi(a_{1234};q)_\infty}{(q,a_{12},a_{13}
,a_{14},a_{23},a_{24},a_{34};q)_\infty}{,}
\label{AWint}
\end{eqnarray}
{where $\max(|a_1|,\ldots,|a_4|)<1$}.
Note the Askey--Wilson norm for $n=0$ is equal to the 
evaluation of the Askey--Wilson integral \eqref{AWint}.
The second is the Nassrallah--Rahman integral 
(in symmetrical form) 
\moro{\cite[(6.3.9)]{GaspRah}}
which generalizes the Askey--Wilson integral, 
namely \cite[(3.1)]{Rahman86prodctsqJ}.
Let {${\bf a}:=\{a_1,a_2,a_3,a_4,a_5\}$}. Then
\begin{eqnarray}
&&\hspace{-0.8cm}
\int_0^\pi\frac{(\expe^{\pm 2i\theta},
\lambda\expe^{\pm i\theta};q)_\infty}
{({\bf a}\expe^{\pm i\theta};q)_\infty}
{\mathrm d}\theta
=\frac{2\pi(\lambda{\bf a},
\lambda ^{-1}a_{12345};q)_\infty}
{(q,a_{12},\ldots,
a_{45},\lambda ^2;q)_\infty}
\Whyp87{\frac{\lambda ^2}{q}}{
\frac{\lambda}{{\bf a}}
}
{q,\frac{a_{12345}}{\lambda}}{,}
\label{NRint}
\end{eqnarray}
{where $\max(|a_1|,\ldots,|a_5|)<1$ \boro{and $|a_{12345}|<|\lambda|$}}.
The Nassrallah--Rahman integral \eqref{NRint} becomes the Askey--Wilson 
integral \eqref{AWint} for $\lambda =a_5.$ 
}

{
The third is the Rahman integral
\cite[Exercise 6.7]{GaspRah} which generalizes
the Nassrallah--Rahman integral.
Let {${\bf a}:=\{a_1,a_2,a_3,a_4,a_5,a_6\}$}. Then
\begin{eqnarray}
&&\hspace{-0.8cm}\int_0^\pi\frac{(\expe^{\pm 2i\theta},\lambda
\expe^{\pm i\theta},\mu\expe^{\pm i\theta};q)_\infty}
{({\bf a}\expe^{\pm i\theta}
;q)_\infty}
{\mathrm d}\theta=\frac{2\pi}{(q,a_{12},\ldots,
a_{56};q)_\infty}\nonumber\\
&&\hspace{0.0cm}\times\Biggl(
\frac{(
\lambda{\bf a},
\frac{\mu}{{\bf a}}
;q)_\infty}
{(\lambda^2,\mu/\lambda;q)_\infty}\Whyp{10}9{\frac{\lambda^2}
{q}}{\frac{\lambda\mu}{q},
\frac{\lambda}{{\bf a}}
}
{q,q}
+\frac{(
\mu{\bf a},
\frac{\lambda}{\bf a}
;q)_\infty}
{(\mu^2,\lambda/\mu;q)_\infty}\Whyp{10}9{\frac{\mu^2}{q}}{
\frac{\lambda\mu}{q},
\frac{\mu}{\bf a}
}{q,q}
\Biggr),
\label{Rint}
\\
&&=\frac{2\pi}{(q,a_{12},\ldots,
a_{56};q)_\infty}\II{\lambda\boro{;}\mu}
\frac{(
\lambda{\bf a},
\frac{\mu}{{\bf a}}
;q)_\infty}
{(\lambda^2,\mu/\lambda;q)_\infty}\Whyp{10}9{\frac{\lambda^2}
{q}}{\frac{\lambda\mu}{q},
\frac{\lambda}{{\bf a}}
}
{q,q},
\end{eqnarray}
where 
$\lambda\mu=a_{123456}$ 
{and $\max(|a_1|,\ldots,|a_6|)<1$}.
Given that $\lambda\ne 0$, then if 
$\mu=a_6\to 0$, then the Rahman integral \eqref{Rint} becomes 
the Nassrallah--Rahman integral \eqref{NRint}.
See \cite{Rahman88} for some other interesting 
limits of the Rahman integral \eqref{Rint}.
}

\medskip
{Some other important integrals related
to basic hypergeometric functions 
are the 
$q$-beta integrals. 
The first one we 
mention is 
due to Askey--Roy 
\cite[(2.8)]{AskeyRoy86}
\begin{equation}
\int_{-\pi}^\pi
\frac{
\big((f{c},\frac{q}{f}d)\frac{\sigma}{z},
(\frac{f}{d},\frac{q}{fc})\frac{z}{\sigma};q\big)_\infty}
{\big((c,d)\frac{\sigma}{z},
(a,b)\frac{z}{\sigma};q\big)_\infty}
{\mathrm d}\psi=2\pi
\frac{
{\vartheta(f,f\frac{c}{d};q)}
(abcd;q)_\infty}
{(q,ac,ad,bc,bd;q)_\infty},
\label{ARint}
\end{equation}
where $z=\expe^{i\psi}$, 
$\max(|q|,|a{/\sigma}|,|b{/\sigma}|,|{\sigma}c|,|{\sigma}d|)<1$ and
$cdf\ne 0$. The second important
$q$-beta integral 
is due to Gasper \cite[(1.8)]{Gasper1989}, namely
\begin{equation}
\int_{-\pi}^\pi
\frac
{((f{c},\frac{q}{f}d)\frac{\sigma}{z},
(\frac{f}{d},\frac{q}{fc},abcde)\frac{z}{\sigma};q)_\infty}
{((c,d)\frac{\sigma}{z},
(a,b,e)\frac{z}{\sigma};q)_\infty}
{\mathrm d}\psi
=2\pi
\frac{
{\vartheta(f,f\frac{c}{d};q)}
(abcd,bcde,acde;q)_\infty}
{(q,ac,ad,bc,bd,ce,de;q)_\infty},
\label{Gint}
\end{equation}
where $z=\expe^{i\psi}$, 
$\max(|q|,|a{/\sigma}|,|b{/\sigma}|,|{\sigma} c|,|{\sigma} d|,|e{/\sigma}|)<1$ and
$cdf\ne 0$.
This integral extends the
Askey--Roy integral
\eqref{ARint} and reduces to it
when $e$ is set to $0$.
{Note that \eqref{Gint} reduces to an 
expression equivalent to 
\cite[(1.8)]{Gasper1989} by taking $\sigma\to 1$ and $f\mapsto f/c$.}
}
{
\begin{rem}
Note that it is the special choice of numerator
parameter behavior in the Askey--Roy and Gasper
integrals which allow one to obtain these
closed-form infinite product representations.
We will return to this in {Theorem} 
\ref{gascaru}.
\end{rem}
}
\section{{Integral representations for basic hypergeometric functions}}
{Here we present a result which follows
by contour integration of products
and quotients of $q$-gamma functions
multiplied by integer powers of a 
complex exponential. Much of the
derivations presented here follow the
pioneering work of Bailey
\cite[Chapter 8]{Bailey64}, his student
Slater 
\cite[Chapters 5 and 7]{Slater66}
and especially from such
works of Askey and Roy \cite{AskeyRoy86}, 
{Nassrallah} and Rahman
\cite{NassrallahRahman85},
Rahman \cite{Rahman10W986},
Gasper \cite{Gasper1989}, and
Gasper and Rahman who
carefully reviewed early preliminary
results as well as deriving fundamental extensions
in \cite[Chapters 4 and 6]{GaspRah}.
}
The following 
theorem is a straightforward generalization 
of Corollary 2.4 in
\cite{IsmailStanton15},
and essentially a restatement of 
\cite[(4.10.5-6)]{GaspRah} using the 
van de Bult--Rains notation
\eqref{topzero}, \eqref{botzero} for
basic hypergeometric series with vanishing
numerator
or denominator parameters.

{
\begin{thm} \label{gascard}
Let  $q\in\CCdag$, $m\in \mathbb Z$,
$t\in \CCast$,
\moro{$\sigma\in(0,\infty)$,}
${\bf a}:=\{a_1,\ldots,a_A\}$, 
${\bf b}:=\{b_1,\ldots,b_B\}$, 
${\bf c}:=\{c_1,\ldots,c_C\}$, 
${\bf d}:=\{d_1,\ldots,d_D\}$ be
sets of complex numbers with cardinality 
$A, B, C, D\in\mathbb N_0$ (not all zero) respectively with
$|c_k|<\sigma/|t|$,
$|d_l|<1/\sigma$,
for any 
$a_i, b_j, c_k, d_l\in\CC$ elements of 
${\bf a}, {\bf b}, {\bf c}, {\bf d}$,
and $z=\expe^{i\psi}$.
Define
\begin{eqnarray}
&&\hspace{-1.7cm}G_{m,t}:=G_{m,t}({\bf a},{\bf b},{\bf c},{\bf d};\sigma,q)
:=\frac{(q;q)_\infty}{2\pi}\left(\frac{\sqrt{t}}{\sigma}\right)^m\int_{-\pi}^\pi
\frac{({\bf b}\frac{\sigma}{z}, t{\bf a}\frac{z}{\sigma};q)_\infty}
{({\bf d}\frac{\sigma}{z}, t{\bf c}\frac{z}{\sigma};q)_\infty}
\expe^{im\psi}{\mathrm d}{\psi},
\label{gascardeq}
\end{eqnarray}
such that the integral exists. Then
\begin{equation}
G_{m,t}({\bf a},{\bf b},{\bf c},{\bf d};\sigma,q)
=G_{-m,t}({\bf b},{\bf a},{\bf d},{\bf c};\sigma,q),
\label{Gmt0}
\end{equation}
if $|c_k|,|d_l|<\min\{1/\sigma, 
\sigma/|t|\}$.
Furthermore, let $td_lc_k\not\in\Omega_q$.
If $D\ge B$,  $d_l/d_{l'}\not\in\Omega_q$, $l\ne l'$, then
\begin{eqnarray}
\label{Gmt1}
&&\hspace{-0.6cm}G_{m,t}\!=\!t^\frac{m}{2}
\sum_{k=1}^D
\frac{(td_k{\bf a},{\bf b}/d_k;q)_\infty d_k^m}
{(td_k{\bf c},{\bf d}_{[k]}/d_k;q)_\infty}
\nqphyp{B+C}{A+D-1}{C-A}
\left(\!\!\begin{array}{c}
td_k{\bf c},q d_k/{\bf b}\\
td_k{\bf a},q d_k/{\bf d}_{[k]}
\end{array}
;q,q^m(qd_k)^{D-B}\frac{b_1\cdots b_B}{d_1\cdots d_D}\right)\!,
\end{eqnarray}
and/or if $C\ge A$, $c_k/c_{k'}\not\in\Omega_q$, $k\ne k'$, then
\begin{eqnarray}
&&\hspace{-0.55cm}G_{m,t}=
\frac{1}{t^{\frac{m}{2}}}
\sum_{k=1}^C\!\frac{(\,tc_{k}{\bf b},
{\bf a}/c_{k};q)_\infty c_k^{-m}
}
{(tc_k{\bf d},{\bf c}_{[k]}/c_k ;q)_\infty}
\nqphyp{A+D}{{B+C-1}}{{D-B}}
\!\!\left(\!\!\begin{array}{c} tc_k\,{\bf d}, q c_{k}/
{\bf a}\\ tc_k\,{\bf b},q c_k/{\bf c}_{[k]} \end{array}\!\!\!;q,
q^{-m}(q c_k)^{C-A}\frac{a_1\cdots a_A}{c_1\cdots c_C}\right)\!,
\label{Gmt2}
\end{eqnarray}
\noindent where the nonterminating basic hypergeometric series 
in \eqref{Gmt1} 
(resp.~\eqref{Gmt2}) 
is  entire if $D>B$ 
(resp.~$C>A$), convergent for
$|q^mb_1\cdots b_B|<|d_1\cdots d_D|$ if $D=B$ 
(resp. $|q^{-m}a_1\cdots a_A|<|c_1\cdots c_C|$ if $C=A$), 
and divergent otherwise.
\end{thm}
}
{
\begin{proof}
{
We obtain the integral expression for 
$G_{m,t}$ \eqref{gascardeq} by 
starting with \cite[(4.9.3)]{GaspRah}
and replacing the sets of parameters 
${\bf a}$, ${\bf b}$, ${\bf c}$, ${\bf d}$ with ${\bf a} t/\sigma$, ${\bf b} \sigma$,
${\bf c} t/\sigma$, ${\bf d} \sigma$.
The relation \eqref{Gmt0} follows
by replacing $\psi$ with $-\psi$
in \eqref{gascardeq}.  
To produce \eqref{Gmt1} and \eqref{Gmt2},
in Gasper \& Rahman \cite[(4.10.5-6)]{GaspRah}, make
the above parameter replacements
and use the van de Bult--Rains notation \eqref{topzero}, \eqref{botzero}.
Note that \eqref{Gmt1}, \eqref{Gmt2}
reduce to \cite[(4.10.5-6)]{GaspRah} 
when $t = \sigma = 1.$ As mentioned
in \cite[\S 4.9]{GaspRah}, these integrals
were used in Slater \cite[Chapter 5]{Slater66} with $m=0,1$.
}
\end{proof}
}
{
\begin{rem}
Note that in the case where the arguments of the basic hypergeometric functions in 
\eqref{Gmt1}, \eqref{Gmt2} are greater
than unity, the integral representations
for $G_{m,t}$, when convergent, may provide an
analytic continuation for these basic
hypergeometric functions.
{
For the integrals in \eqref{gascardeq} with respect to the variable of integration $\psi$ the line of integration from $-\pi$ to $\pi$ in the $\psi$-plane would have to be replaced by a suitably deformed line in the $\psi$-plane separating the sequences of poles that have infinitely many poles in the upper half $\psi$-plane from those with infinitely many poles in the lower half $\psi$-plane.}
\end{rem}
}


\begin{rem} 
{
Observe that in the case where \eqref{gascardeq} can be written as \eqref{Gmt1} or \eqref{Gmt2} 
(e.g., $m\in\mathbb Z$)
then 
$G_{m,t}$ 
does not depend on $\sigma$.
}
\end{rem}

\subsection{Argument $q$ applications of Theorem \ref{gascard}}

Note that when identifying integral representations for 
basic hypergeometric series, Theorem \ref{gascard}
is extremely useful. However, in applications
even though one may use it to identify
the parameters of a symmetric sum of
basic hypergeometric functions, the restriction
on the argument is often problematic. On the
other hand, with the special choice of parameters
given in the following corollaries which 
leads to an argument $q$, the ability
to tie parameters to specific basic hypergeometric
functions is greatly enhanced.
Here we present some generalized results 
which gives the symmetric sum of 
two terms each containing a basic hypergeometric function with argument $q$.

\begin{thm}\label{gascaru}
Let {$q\in\CCdag$,} ${\bf a}:=\{a_1,\ldots,a_A\}$, 
${\bf c}:=\{c_1,\ldots,c_C\}$,
be sets of complex numbers with cardinality 
$A,C\in\mathbb N_0$ (not {both} zero) respectively, 
${\bf d}:=\{d_1,d_2\}$, 
$c_kd_l\not\in\Omega_q$,
$z=\expe^{i\psi}$,
$\sigma\in(0,\infty)$, $d_1, d_2\in\CCast$,
such that
$|c_k|<\sigma$, $|d_1|,|d_2|<1/\sigma$, 
for any 
$c_k\in {\bf c}$.
Define
\begin{eqnarray}
&&
\hspace{-0.7cm}H({\bf a},{\bf c},{\bf d};q):=\II{d_1\boro{;}d_2}
\frac{(d_1{\bf a};q)_\infty}
{\left(\frac{d_2}{d_1},d_1{\bf c};q\right)_{\!\infty}}
\!\qphyp{C}{A+1}{C-A-2}
{d_1{\bf c}}
{d_1{\bf a},q d_1/d_2}
{q,q}\\
&&\hspace{-0.5cm}
=\!
\frac{(d_1{\bf a};q)_\infty}
{\left(\frac{d_2}{d_1},d_1{\bf c};q\right)_{\!\infty}}
\!\qphyp{C}{A+1}{C-A-2}
{d_1{\bf c}}
{d_1{\bf a},q d_1/d_2}
{q,q}\!+\!
\frac{(d_2{\bf a};q)_\infty}
{\left(\frac{d_1}{d_2},d_2{\bf c};q\right)_{\!\infty}}
\!\qphyp{C}{A+1}{C-A-2}
{d_2{\bf c}}
{d_2{\bf a},q d_2/d_1}
{q,q}\!,\label{Hdefn}
\end{eqnarray}
where $d_l/d_{l'}\not\in\Omega_q$, $l\ne l'$, and if $C\ge A+2$,
\begin{eqnarray}
&&\hspace{-1.7cm}J({\bf a},{\bf c},{\bf d};f,q):=
\sum_{k=1}^C
\frac{
{\vartheta(fc_kd_1,\frac{f}{c_kd_2};q)}
({\bf a}/c_k
;q)_\infty}
{(c_k{\bf d},
{\bf c}_{[k]}/c_k;q)_\infty}\nonumber\\
&&\hspace{3.9cm}\times\qhyp{A+2}{C-1}
{c_k{\bf d},
qc_k/{\bf a}}{qc_k/{\bf c}_{[k]}}
{q,\frac{q(qc_k)^{C-A-2}a_1\cdots a_A}
{d_1d_2c_1\cdots c_C}},
\label{Jdefn}
\end{eqnarray}
where $c_k/c_{k'}\not\in\Omega_q$, $k\ne k'$, 
and ${}_{A+2}\phi_{C-1}$ is convergent
for $C=A+2$ if 
$|qa_1\cdots a_A|<|d_1d_2c_1\cdots c_C|$,
and is an entire function if $C>A+2$.
Then
\begin{eqnarray}
&&\hspace{-1.6cm}\int_{-\pi}^\pi
\frac{((fd_1,\frac{q}{f}d_2)
\frac{\sigma}{z}, 
(\frac{f}{d_2},\frac{q}{fd_1},{\bf a})\frac{z}{\sigma};q)_\infty}
{((d_1,d_2)\frac{\sigma}{z}, {\bf c}\frac{z}{\sigma};q)_\infty}
{\mathrm d}{\psi}\label{qqform}
=\frac{2\pi
{\vartheta(f,f\frac{d_1}{d_2};q)}
}{(q;q)_\infty}
H({\bf a},{\bf c},{\bf d};q)
\\
&&\hspace{4.96cm}=\frac{2\pi}{(q;q)_\infty}J({\bf a},{\bf c},{\bf d};f,q),\quad (C\ge A+2)\moro{,}
\label{jjform}
\end{eqnarray}
\moro{and none of the arguments of the modified theta functions are equal to some $q^m$, $m\in{\mathbb Z}$.}
\end{thm}
{
\begin{proof}
Starting with \eqref{gascardeq}, \eqref{Gmt1}
with $m=0$, $t=1$,
and substituting the 
parameters as in the integrand of 
\eqref{qqform}, two of the numerator parameters
cancel with two of the denominator parameters
and the argument of the basic hypergeometric 
function reduces to $q$. Noting that the nonterminating basic hypergeometric series are either of the form ${}_C\phi_{C-1}$ or ${}_{A+2}\phi_{A+1}$, depending on whether $C-A-2$ is negative or positive respectively,
the series is convergent for
$|q|<1$, and this produces the right-hand side of \eqref{qqform}.
One produces \eqref{jjform} by starting with
\eqref{gascardeq}, \eqref{Gmt2} with $m=0$, 
$t=1$ and using the convergence properties
of the nonterminating basic hypergeometric series described
in Theorem \ref{gascard}.
This completes the proof.
\end{proof}
}

{
\begin{cor}
Let {$q\in\CCdag$,} ${\bf b}:=\{b_1,\ldots,b_B\}$, 
${\bf d}:=\{d_1,\ldots,d_D\}$,
be sets of complex numbers with cardinality 
$B,D\in\mathbb N_0$ (not both zero) respectively, 
${\bf c}:=\{c_1,c_2\}$, 
$z=\expe^{i\psi}$,
$\sigma\in(0,\infty)$, $c_1, c_2\in\CCast$,
$f\in \CCast\setminus\{1\}$, 
such that
$|d_l|<1/\sigma$, $|c_1|,|c_2|<\sigma$, 
for any 
$d_l\in {\bf d}$.
Then
\begin{eqnarray}
&&\hspace{-1.6cm}\int_{-\pi}^\pi
\frac{(
(\frac{f}{c_2},\frac{q}{fc_1},{\bf b})\frac{\sigma}{z},
(fc_1,\frac{q}{f}c_2)\frac{z}{\sigma}
;q)_\infty}
{({\bf d}\frac{\sigma}{z},
(c_1,c_2)\frac{z}{\sigma} 
;q)_\infty}
{\mathrm d}{\psi}\label{qqform2}
=\frac{2\pi
{\vartheta(f,f\frac{c_1}{c_2};q)}
}{(q;q)_\infty}
H({\bf b},{\bf d},{\bf c};q)
\\
&&\hspace{4.85cm}=\frac{2\pi}{(q;q)_\infty}J({\bf b},{\bf d},{\bf c};f,q),\quad (D\ge B+2),
\label{jjform2}
\end{eqnarray}
and ${}_{B+2}\phi_{D-1}$ in $J$ is convergent
for $D=B+2$ if 
$|qb_1\cdots b_B|<|c_1c_2d_1\cdots d_D|$,
and is an entire function if $D>B+2$\moro{,}
\moro{and none of the arguments of the modified theta functions are equal to some $q^m$, $m\in{\mathbb Z}$.}
\end{cor}
}
{
\begin{proof}
As in the proof of {Theorem} \ref{gascaru},
start with Theorem \ref{gascard} with $m=0$, $t=1$.
Use \eqref{gascardeq}, \eqref{Gmt2},
and substitute the 
parameters as in the integrand of 
\eqref{qqform2}. Noting that the nonterminating basic hypergeometric series are either of the form ${}_D\phi_{D-1}$ or ${}_{B+2}\phi_{B+1}$, depending on whether $D-B-2$ is negative or positive respectively,
the series is convergent for
$|q|<1$, and this produces the right-hand side
of \eqref{qqform2}.
One produces \eqref{jjform2} by starting with
\eqref{qqform2} with
\eqref{gascardeq}, \eqref{Gmt1} and using the convergence properties
of nonterminating basic hypergeometric series.
This completes the proof.
\end{proof}
}

\noindent{
\begin{thm}
\label{qintqone}
Let $H$, ${\bf a}$, ${\bf c}$, ${\bf d}$, $q$ be defined
as in {Theorem} \ref{gascaru} and $s\in\CCast$,
\moro{$d_2{d_1}^{-1}\ne q^m$, $m\in\Z$}. Then
\begin{eqnarray}
&&\hspace{-1cm}H({\bf a},{\bf c},{\bf d};q)=\frac{\sqrt{\frac{d_2}{d_1}}}
{(1-q)s(q;q)_\infty
{\vartheta(\frac{d_2}{d_1};q)}
}
\int_{s\sqrt{\frac{d_2}{d_1}}}^{s\sqrt{\frac{d_1}{d_2}}}
\frac{((q\sqrt{d_1/d_2
},q\sqrt{
d_2/d_1
},{\bf a}\sqrt{d_1d_2})\frac{u}{s};q)_\infty}
{({\bf c}\sqrt{d_1d_2}\frac{u}{s};q)_\infty}{\mathrm d}_qu,
\label{qintH}
\end{eqnarray}
which is symmetric in $\{d_1,d_2\}$, as in
\eqref{qJacsym}.
\end{thm}
}
{
\begin{proof}
Start with the $q$-integral on the right-hand side
of \eqref{qintH} using the definition \eqref{qint}.
Replacing the Pochhammer symbols using \eqref{infPochdefnb} identifies 
the argument $q$ basic hypergeometric series in question. 
Then pulling out common factors, using \eqref{infone}
and comparing with \eqref{Hdefn} derives \eqref{qintH}. 
The symmetry in $\{d_1,d_2\}$ is clear from \eqref{Hdefn}. 
This completes the proof.
\end{proof}
}

{
Note that from {Theorem} \ref{gascaru} and Theorem
\ref{qintqone}, we arrive at an interesting relation
of a definite integral with a $q$-integral.
}
{
\begin{cor}\label{Int-Qint}
Let ${\bf a}$, ${\bf c}$, ${\bf d}$, $q$, $s$, $f$, $\sigma$, 
$z=\expe^{i\psi}$ be defined
as in {Theorem} \ref{gascaru} and Theorem \ref{qintqone}. Then
\begin{eqnarray}
&&\hspace{-0.6cm}\int_{-\pi}^\pi
\frac{((fd_1,\frac{q}{f}d_2)
\frac{\sigma}{z}, 
(\frac{f}{d_2},\frac{q}{fd_1},{\bf a})\frac{z}{\sigma};q)_\infty}
{((d_1,d_2)\frac{\sigma}{z}, {\bf c}\frac{z}{\sigma};q)_\infty}
{\mathrm d}{\psi}
=\frac{2\pi\sqrt{\frac{d_2}{d_1}}
{\vartheta(f,f\frac{d_1}{d_2};q)}
}{(1-q)s(q,q;q)_\infty
{\vartheta(\frac{d_2}{d_1};q)}
}\nonumber\\
&&\hspace{5cm}\times\int_{s\sqrt{\frac{d_2}{d_1}}}^{s\sqrt{\frac{d_1}{d_2}}}
\frac{((q\sqrt{
d_1/d_2
},q\sqrt{
d_2/d_1
},{\bf a}\sqrt{d_1d_2})\frac{u}{s};q)_\infty}
{({\bf c}\sqrt{d_1d_2}\frac{u}{s};q)_\infty}{\mathrm d}_qu.
\end{eqnarray}
\end{cor}
}
{
\begin{proof}
Comparing {Theorem} \ref{gascaru} with Theorem
\ref{qintqone} completes the proof.
\end{proof}
}


A useful consequence of this formula is 
given in \cite[Exercise 4.4]{GaspRah},
which is an application of \eqref{qqform} with $C=4$, {$A=2$}. It takes advantage of 
Bailey's transformation of a very-well-poised ${}_8W_7$ \cite[(17.9.16)]{NIST:DLMF} and is {given} as follows\moro{:}
\begin{eqnarray}
&&\hspace{-0.1cm}\int_{-\pi}^\pi
\frac{((cf,\frac{qd}{f})\frac{\sigma}{z},
(\frac{f}{d},\frac{q}{fc},k,\frac{abcdgh}{k})
\frac{z}{\sigma};q)_\infty}
{((c,d)\frac{\sigma}{z},(a,b,g,h)\frac{z}{\sigma};q)_\infty}{\mathrm d}\psi
=\frac{2\pi
{\vartheta(f,f\frac{c}{d};q)}
(kc,kd,acdg,bcdg,cdgh,\frac{abcdh}{k};q)_\infty}
{(q,ac,ad,bc,bd,cg,dg,ch,dh,kcdg;q)_\infty}
\nonumber\\
&&\hspace{7.5cm}\times\Whyp87{\frac{kcdg}{q}}{cg,dg,\frac{k}{a},\frac{k}{b},\frac{k}{h}}{q,\frac{abcdh}{k}},
\label{fourfour}
\end{eqnarray}
where $z=\expe^{i\psi}$,
$\max(|a|,|b|,|c|,|d|,|g|,|h|)<1$, and $|abcdh|<|k|$.
{Note that if $h=k$ 
and $g\mapsto e$ then \eqref{fourfour}
reduces to Gasper's integral \eqref{Gint}.}

\medskip
{
Using
Theorem \ref{gascard}
one can
derive the following generalization of Rahman's integral \eqref{Rint} 
which does not include the constraint
$\lambda\mu=a_{123456}$.
}
{
\begin{thm}Let ${\bf a}:=\{a_1,\ldots,a_6\}$, $a_1,\ldots,a_6,\lambda,\mu\in\CCast$,
$q\in\CCdag$. Then
\begin{eqnarray}
&
\int_{-\pi}^\pi\frac{(\expe^{\pm 2i\psi},\lambda \expe^{\pm i\psi}
,\mu\expe^{\pm i\psi};q)_\infty}{(
{\bf a}\expe^{\pm i\psi}
;q)_\infty}&{\mathrm d}\psi
=\frac{2\pi}{(q;q)_\infty}
\hspace{-3mm}\II{a_1\boro{;a_2},\ldots,a_6}\hspace{-2mm}
\dfrac{\left(a_1^{-2},a_1\lambda,a_1\mu, \dfrac{\lambda}{a_1},\dfrac{\mu}{a_1};q\right)_\infty}
{\left(a_{12},\ldots,a_{16},\dfrac{a_2}{a_1},\ldots
,\dfrac{a_6}{a_1};q\right)_\infty}\nonumber\\&&
\hspace{1.2cm}
\times
{\Whyp{10}{9}{a_1^2}{a_{12},\ldots,a_{16},\frac{q a_1}{\lambda},\frac{q a_1}{\mu}}
{q,\frac{q\lambda\mu}{a_{123456}}}},
\label{genRah}
\end{eqnarray}
where 
$|q\lambda\mu|<|a_{123456}|{<1}$
{and $\max(|a_1|,\ldots,|a_6|)<1$}.
\end{thm}
}
{
\begin{proof}
Starting with the left-hand side of \eqref{genRah}
and applying Theorem \ref{gascard},
{noting $qa_1^2/(\pm q^\frac12 a_1)=\pm q^\frac12 a_1$,}
completes the proof.
\end{proof}
}

{
If we set $\lambda\mu=a_{123456}$ then
\eqref{genRah} specializes to Rahman's
integral \eqref{Rint}. This results in 
the following transformation law for the
symmetrized sum of six {${}_{10}W_{9}$'s}
with argument $q$ being equal to the
symmetrized sum of two ${}_{10}W_9$'s with 
argument $q$.
}

{
\begin{cor}Let ${\bf a}:=\{a_1,\ldots,a_6\}$, $a_1,\ldots,a_6,\lambda,\mu\in\CCast$,
$q\in\CCdag$. Then
\begin{eqnarray}
&&\hspace{0cm}\II{a_1\boro{;a_2},\ldots,a_6}\!\!
\dfrac{\left(a_1^{-2},a_1\lambda,a_1\mu, \dfrac{\lambda}{a_1},\dfrac{\mu}{a_1};q\right)_\infty}
{\left(a_{12},\ldots,a_{16},\dfrac{a_2}{a_1},\ldots
,\dfrac{a_6}{a_1};q\right)_\infty}
{\Whyp{10}{9}{a_1^2}{a_{12},\ldots,a_{16},
\frac{q a_1}
{\lambda},\dfrac{q a_1}{\mu}}{q,q}}\nonumber\\
&&\hspace{3cm}=
2\II{\lambda\boro{;}\mu}
\frac{\left(
\lambda{\bf a},
\dfrac{\mu}{\bf a}
;q\right)_\infty}
{(a_{12},\ldots,a_{56};q)_\infty\left(\lambda^2,\dfrac{\mu}{\lambda};q\right)_\infty}
\Whyp{10}{9}{\frac{\lambda^2}{q}}{
\dfrac{\lambda}{\bf a},
\dfrac{a_{123456}}{q}
}{q,q}.
\end{eqnarray}
\end{cor}
}
{
\begin{proof}
Comparing \eqref{genRah} to \eqref{Rint} and noting that the integrand 
is an even function of $\psi$,
{and noting $qa_1^2/(\pm q^\frac12 a_1)=\pm q^\frac12 a_1$,}
completes the proof.
\end{proof}
}

{
Using 
Theorem \ref{gascard}
we can find an alternative
expression for the Nassrallah--Rahman integral
\eqref{NRint} as a symmetrized sum of five {${}_{8}W_7$'s}.
}
{
\begin{thm}Let 
${\bf a}:=\{a_1,\ldots,a_5\}$,
$a_1,\ldots,a_5,\lambda,\mu\in\CCast$,
$q\in\CCdag$. Then
\begin{eqnarray}
&
\int_{-\pi}^\pi
\frac{(\expe^{\pm 2i\psi},\lambda \expe^{\pm i\psi};q)_\infty}
{({\bf a}\expe^{\pm i\psi}
;q)_\infty}
{\mathrm d}\psi
=&\frac{2\pi}{(q;q)_\infty}\II{a_1\boro{;a_2},\ldots,a_5}\hspace{-2mm}
\frac{\left(a_1^{-2},a_1\lambda, \dfrac{\lambda}{a_1};q
\right)_\infty}{\left(a_{12},\ldots,a_{15},\dfrac{a_2}
{a_1},\ldots,\dfrac{a_5}{a_1};q\right)_\infty}
\nonumber\\
&&\hspace{1.5cm}
\times{\Whyp{8}{7}{a_1^2}{a_{12},\ldots,a_{15},
\dfrac{q a_1}{\lambda}}{q,\frac{q\lambda}{a_{12345}}}},
\label{altNRint}
\end{eqnarray}
where 
$|q\lambda|<|a_{12345}|{<1}$
{and $\max(|a_1|,\ldots,|a_5|)<1$}.
\end{thm}
}
{
\begin{proof}
Starting with the left-hand side of \eqref{altNRint}
and applying 
Theorem \ref{gascard}
completes the proof.
\end{proof}
}

{By 
comparing the above expression for the Nassrallah--Rahman 
integral to \eqref{altNRint}, one can obtain the following 
transformation of a symmetrized sum of five {${}_{8}W_7$'s} is equal
to a symmetric ${}_8W_7$.
}

{
\begin{cor}Let 
${\bf a}:=\{a_1,\ldots,a_5\}$,
$a_1,\ldots,a_5,\lambda\in\CCast$,
$q\in\CCdag$. Then
\begin{eqnarray}
&&\hspace{-1cm}\II{a_1\boro{;a_2},\ldots,a_5}
\frac{\left(\lambda a_1,a_1^{-2},\dfrac{\lambda}{a_1};q\right)_\infty}
{\left(a_{12},\ldots,a_{15},\dfrac{a_2}{a_1},\ldots,\dfrac{a_5}
{a_1};q\right)_\infty}
{\Whyp{8}{7}{a_1^2}{a_{12},\ldots,a_{15},
\frac{q a_1}{\lambda}}{q,\frac{q\lambda}{a_{12345}}}}\nonumber\\
&&\hspace{5.3cm}=\frac{{2}\left(
\lambda{\bf a},
\dfrac{a_{12345}}{\lambda};q\right)_\infty}
{(a_{12},\ldots,a_{45},\lambda^2;q)_\infty}
\Whyp87{\frac{\lambda^2}{q}}{
\frac{\lambda}{\bf a}
}
{q,\frac{a_{12345}}{\lambda}},
\label{compNR}
\end{eqnarray}
where 
$|q|<\left|\dfrac{a_{12345}}{\lambda}\right|<1$.
\end{cor}
}
{
\begin{proof}
Comparing \eqref{altNRint} to \eqref{NRint} and noting that the 
integrand is an even function of $\psi$ completes the proof.
\end{proof}
}

{By taking $\lambda=a_5$
one reduces the Nassrallah--Rahman integrals 
\eqref{NRint}, \eqref{altNRint} to the Askey--Wilson integral
\eqref{AWint} (the ${}_8W_7$ becomes unity). Comparing these
limit expressions produces the following
nonterminating summation formula
which relates a symmetric sum of 
four {${}_6W_5$'s} to an infinite product {that} we now give.
}

\begin{cor}Let $a_1,\ldots,a_4\in\CCast$,
$q\in\CCdag$, \moro{and none of the arguments of the modified theta functions are equal to some $q^m$, $m\in{\mathbb Z}$,
and none of infinite $q$-shifted factorials vanish}. Then
\begin{eqnarray}
&&\hspace{-3cm}\II{a_1;a_2,\boro{a_3},a_4}
\frac{(a_1^{-2};q)_\infty}
{\left(a_{12},a_{13},a_{14},\dfrac{a_2}{a_1}
,\dfrac{a_3}{a_1},\dfrac{a_4}{a_1};q\right)_\infty}
{\Whyp65{a_1^2}{a_{12},a_{13},a_{14}
}{q,\frac{q}{a_{1234}}}}
\nonumber\\
&&\hspace{-1cm}
{=\frac{(a_{1234};q)_\infty}
{(a_{12},a_{13},a_{14},a_{23},a_{24},a_{34};q)_\infty}
\II{a_1;a_2,\boro{a_3},a_4}
\frac{{\vartheta}(a_1^{-2},a_{23},a_{24},a_{34};q)}{{\vartheta}(a_{1234},\frac{a_2}{a_1},
\frac{a_3}{a_1},\frac{a_4}{a_1};q)}}
\nonumber\\
&&\hspace{2cm}=\frac{{2}(a_{1234};q)_\infty}
{(a_{12},a_{13},a_{14},a_{23},a_{24},a_{34};q)_\infty},
\end{eqnarray}
where $|q|<\left|a_{1234}\right|<1$.
\end{cor}
{
\begin{proof}
Setting $\lambda=a_5$ in \eqref{altNRint}, and
comparing with \eqref{AWint} completes the proof.
\end{proof}
}

{
\begin{rem}
Note that for the ${}_8W_7(a;b,c,d,e,f;q,z)$'s
which appear in this subsection, instead of 
the argument being $q^2a^2/(bcdef)$ it is
$-q^2a^2/(bcdef)$. Compare with 
Bailey's transformation of a very-well poised
${}_8W_7$ \cite[(17.9.16)]{NIST:DLMF}. So these
${}_8W_7$'s cannot be written as a sum of two
balanced ${}_4\phi_3$'s.
\end{rem}
}

{Some other applications of {Theorem} \ref{gascaru} arise when one encounters 
a sum of two basic hypergeometric functions with
argument $q$. In this case, you} are almost certainly guaranteed 
to be able to find a corresponding
integral representation. Below
we present some examples of this.
First we present two integral representations
of a nonterminating ${}_2\phi_1$.
Note that other integral representations for
the arbitrary ${}_2\phi_1$ have been
presented such as Watson's contour integral {(see \cite[(4.2.2)]{GaspRah} for more details)}
\begin{equation}
\qhyp21{a,b}{c}{\moro{q,}z}=-\frac{1}{2i}
\frac{(a,b;q)_\infty}{(q,c;q)_\infty}
\int_{-i\infty}^{i\infty}
\frac{((q,c)q^s;q)_\infty}
{((a,b)q^s;q)_\infty}\frac{(-z)^s}{\sin(\pi s)}{\mathrm d}s,
\end{equation}
where $\pm i\infty:=\pm\lim_{x\uparrow\infty}ix$,
where $x\in(0,\infty)$.

\begin{cor}Let $a,b,c,z\in\CCast$, such that $|z|<1$, $q\in\CCdag$, $\tau\in(0,1)$, 
$w=\expe^{i\eta}$. Then
\begin{eqnarray}
&&\hspace{-1cm}\qhyp21{a,b}{c}{q,z}\nonumber\\
\label{2phi1a}&&\hspace{0cm}=
\frac{(q,a,\frac{c}{b},\frac{abz}{c};q)_\infty}{2\pi
{\vartheta(f,f\frac{c}{bz};q)}
}
\int_{-\pi}^\pi
\frac{((f\sqrt{\frac{c}{bz}},\frac{q}{f}\sqrt{\frac{bz}{c}})
\frac{\tau}{w},(f\sqrt{\frac{c}{bz}},\frac{q}{f}\sqrt{\frac{bz}
{c}},\sqrt{bcz})\frac{w}{\tau};q)_\infty}
{((\sqrt{\frac{c}{bz}},\sqrt{\frac{bz}{c}})\frac{\tau}{w},(\sqrt{\frac{bz}{c}}a,
\sqrt{\frac{cz}{b}})\frac{w}{\tau};q)_\infty}
{\mathrm d}\eta
\\
&&\label{2phi1b}\hspace{0cm}=
\frac{(q,a,b,\frac{c}{a},\frac{c}{b},\frac{abz}{c};q)_\infty}
{2\pi
{\vartheta(f,f\frac{ab}{c};q)}
(c;q)_\infty
}
\int_{-\pi}^\pi
\frac{((f\sqrt{\frac{ab}{c}},\frac{q}{f}\sqrt{\frac{c}{ab}})\frac{\tau}{w},(f\sqrt{\frac{ab}{c}},\frac{q}{f}\sqrt{\frac{c}{ab}})\frac{w}{\tau};q)_\infty}
{((\sqrt{\frac{ab}{c}},\sqrt{\frac{c}{ab}})\frac{\tau}{w},(\sqrt{\frac{ac}{b}},\sqrt{\frac{bc}{a}},\sqrt{\frac{ab}{c}}z)\frac{w}{\tau};q)_\infty}{\mathrm d}\eta\moro{,}
\end{eqnarray}
\moro{and none of the arguments of the modified theta functions are equal to some $q^m$, $m\in{\mathbb Z}$.}
\end{cor}
\begin{proof}
For \eqref{2phi1a} start with cf.~\cite[(17.9.3)]{NIST:DLMF}
\begin{eqnarray}
&&\hspace{-0.5cm}\frac{(c;q)_\infty}{(a,\frac{c}{b},\frac{abz}{c};q)_\infty}\qhyp21{a,b}{c}{q,z}\nonumber\\
&&\hspace{1.5cm}
=\frac{(c;q)_\infty}{(a,\frac{c}{b},\frac{bz}{c};q)_\infty}\qphyp{2}{2}{-1}{a,\frac{c}{b}}{c,\frac{qc}{bz}}{q,q}+\frac{(bz;q)_\infty}{(z,\frac{c}{bz},\frac{abz}{c};q)_\infty}\qphyp22{-1}
{z,\frac{abz}{c}}{bz,\frac{qbz}{c}}{q,q},
\end{eqnarray}
then {apply}  \eqref{qqform} with 
\begin{equation}
{\bf a}:=\{\sqrt{bcz}\},\ 
{\bf c}:=\left\{\sqrt{\frac{bz}{c}}a,\sqrt{\frac{cz}{b}}\right\}, \ 
{\bf d}:=\left\{\sqrt{\frac{c}{bz}},\sqrt{\frac{bz}{c}}\right\},
\end{equation}
and therefore $C-A-2=-1$.
For \eqref{2phi1b} start with cf.~\cite[(17.9.3\_5)]{NIST:DLMF}
\begin{eqnarray}
&&\hspace{-0.5cm}\frac{(c;q)_\infty}{(a,b,\frac{c}{a},\frac{c}{b},\frac{abz}{c};q)_\infty}\qhyp21{a,b}{c}{q,z}\nonumber\\
&&\hspace{1.5cm}
=\frac{1}{(a,\frac{c}{b},\frac{bz}{c};q)_\infty}\qphyp{3}{1}{1}{a,b,\frac{abz}{c}}{\frac{qab}{c}}{q,q}+\frac{1}{(z,c,\frac{c}{a},\frac{c}{b};q)_\infty}\qphyp31{1}
{\frac{c}{a},\frac{c}{b},z}{\frac{qc}{ab}}{q,q},
\end{eqnarray}
then {apply} \eqref{qqform} with 
\begin{equation}
{\bf a}:=\emptyset,\ 
{\bf c}:=\left\{\sqrt{\frac{ac}{b}},\sqrt{\frac{bc}{a}},\sqrt{\frac{ab}{c}}\right\}, \ 
{\bf d}:=\left\{\sqrt{\frac{ab}{c}},\sqrt{\frac{c}{ab}}\right\},
\end{equation}
and therefore $C-A-2=1$. This completes the
proof.
\end{proof}

Another example where an integral representation for a 
nonterminating basic hypergeometric function may be found
is for a well-poised ${}_3\phi_2$.
\begin{cor}
\label{3phi2thm}
Let $a,b,c,x\in\CCast$, such that $|qax|<|bc|$, $q\in\CCdag$, $\tau\in(0,1)$, 
$w=\expe^{i\eta}$,
\moro{$f,fx\ne q^m$, $m\in\Z$}. Then
\begin{eqnarray}
&&\hspace{-1cm}\qhyp32{a,b,c}{\frac{qa}{b},\frac{qa}{c}}{q,\frac{qax}{bc}}
=\frac{(q,a,\frac{qa}{bc};q)_\infty}{2\pi
{\vartheta(f,f\boro{x};q)}
(\frac{qa}{b},\frac{qa}{c};q)_\infty}\nonumber\\
&&\hspace{2cm}\times\int_{-\pi}^{\pi}
\frac{(
(
\boro{f\sqrt{x},\frac{q}{f\sqrt{x}}}
)
\frac{\tau}{w},(
\boro{f\sqrt{x},\frac{q}{f\sqrt{x}},\frac{qa\sqrt{x}}{b},\frac{qa\sqrt{x}}{c},ax^\frac32}
\boro{)}\frac{w}{\tau};q)_\infty}{((
\boro{\sqrt{x},\frac{1}{\sqrt{x}}}
)\frac{\tau}{w},(
\boro{\pm\sqrt{ax},\pm\sqrt{qax},\frac{qa\sqrt{x}}{bc}}
)\frac{w}{\tau};q)_\infty}{\mathrm d}\eta.
\end{eqnarray}
\end{cor}
\begin{proof}
Start with cf.~\cite[(III.35)]{GaspRah}, then
\begin{eqnarray}
&&\hspace{-1.5cm}\boro{\frac{(\frac{qa}{b},\frac{qa}{c};q)_\infty}{(a,\frac{qa}{bc};q)_\infty}}\qhyp32{a,b,c}{\frac{qa}{b},\frac{qa}{c}}{q,\boro{\frac{qax}{bc}}}\nonumber\\
&&\hspace{1.5cm}
=\frac{(\frac{qax}{b},\frac{qax}{c};q)_\infty}{(1/x,\frac{qax}{bc};q)_\infty}\qhyp{5}{4}{\pm x\sqrt{a},\pm x\sqrt{qa},\frac{qax}{bc}}{qx,\frac{qax}{b},\frac{qax}{c},\boro{ax^2}}{q,q}\nonumber\\
&&\hspace{3cm}+\frac{(\frac{qa}{b},\frac{qa}{c},ax;q)_\infty}{(a,\frac{qa}{bc},x;q)_\infty}\qhyp54
{\pm\sqrt{a},\pm\sqrt{qa},\frac{qa}{bc}}{\frac{qa}{b},\frac{qa}{c},q/x,ax}{q,q}.
\end{eqnarray}
Applying \eqref{qqform} with 
\begin{equation}
{\bf a}:=\left\{\frac{qa\sqrt{x}}{b},\frac{qa\sqrt{x}}{c},ax^\frac32\right\},\ 
{\bf c}:=\left\{\pm\sqrt{ax},\pm\sqrt{qax},\frac{qa\sqrt{x}}{bc}\right\}, \ 
{\bf d}:=\left\{\sqrt{x},\frac{1}{\sqrt{x}}\right\},
\end{equation}
and therefore $C-A-2=0$, completes the proof.
\end{proof}

\subsection{{Unbalanced symmetrization transformations for basic hypergeometric functions and some of their specializations and limits}}

A direct consequence of {Theorem} \ref{gascaru} is the following 
integral.
\begin{lem}\label{intlem}
Let ${\bf a}:=\{a_1,\ldots,a_A\}$, 
${\bf b}:=\{b_1,\ldots,b_B\}$, 
${\bf c}:=\{c_1,\ldots,c_C\}$, 
${\bf d}:=\{d_1,\ldots,d_D\}$ be
sets of complex numbers with cardinality 
$A, B, C, D\in\mathbb N_0$ (not all zero) respectively, 
$z=\expe^{i\psi}$, $w=\expe^{i\eta}$. 
Let \moro{$\sigma,\tau\in(0,\infty)$,} $t\in \CCast$, so that 
$|b_j|<|d_l|<\min\{\moro{\tau}/|t|,1/\sigma\}$, $|d_la_i|<1/|t|$, 
and $|a_i|<|c_k|<\min\{\sigma/|t|,1/\moro{\tau}\}$ for any 
$a_i, b_j, c_k, d_l$ elements of ${\bf a}, {\bf b}, {\bf c}, 
{\bf d}$ respectively. Then 
\begin{eqnarray}
&&\int_{-\pi}^\pi
\frac{({\bf b}\frac{\sigma}{z}, t\,{\bf a}\frac{z}{\sigma};q)_\infty}
{({\bf d}\frac{\sigma}{z}, t\,{\bf c}\frac{z}{\sigma};q)_\infty}
{\mathrm d}{\psi}
=
\int_{-\pi}^\pi
\frac{({\bf a}\frac{\tau}{w}, t\,{\bf b}\frac{w}{\tau};q)_\infty}
{({\bf c}\frac{\tau}{w}, t\,{\bf d}\frac{w}{\tau};q)_\infty}
{\mathrm d}{\eta}=\frac{2\pi}{(q;q)_\infty} G_{0,t}.
\end{eqnarray}
\end{lem}
{
\begin{proof}
Setting $m=0$ in \eqref{gascardeq}, it is straightforward to 
check the identity by comparing the first summation expression of 
\eqref{gascardeq} to the
the second summation expression of \eqref{gascardeq}. Hence the result holds.
\end{proof}
}
Therefore taking into account the definition of $G_{m,t}$ 
(see expression \eqref{gascardeq}) the following identity holds.

\begin{cor}
\label{corsumDC}
Let ${\bf a}$, ${\bf b}$, ${\bf c}$, ${\bf d}$,
and the other variables {be} as defined as in Lemma \ref{intlem}.
Then one has the following (in general non-balanced)
transformation of symmetrization over variables 
for basic hypergeometric functions:
\begin{eqnarray}
&&\sum_{k=1}^D
\frac{(td_k\,{\bf a},d_k^{-1}\,{\bf b};q)_\infty}
{(td_k{\bf c},d_k^{-1}\,{\bf d}_{[k]};q)_\infty}
\qphyp{B+C}{A+D-1}{A-C}{td_k\,{\bf c},qd_k\,{\bf b}^{-1}}{td_k\,{\bf a}
,q d_k{\bf d}_{[k]}^{-1}}{q,\frac{b_1\cdots b_B}{d_1\cdots d_D}}\nonumber\\
&&=\sum_{k=1}^C
\frac{(tb_k\,{\bf c},b_k^{-1}\,{\bf a};q)_\infty}
{(tc_k{\bf d},c_k^{-1}\,{\bf c}_{[k]};q)_\infty}
\qphyp{A+D}{B+C-1}{B-D}{tc_k\,{\bf d},qc_k\,{\bf a}^{-1}}
{tc_k\,{\bf b},q c_k{\bf c}_{[k]}^{-1}}{q,\frac{a_1\cdots a_A}{c_1\cdots c_C}}.
\end{eqnarray}
\label{megathm}
\end{cor}

\begin{proof}
The identity follows by using Theorem \ref{gascard}
with $m=0$.
\end{proof}


{Now we treat the $t=1$ case which has an extra degree of
symmetry {that} can be exploited.}
{
\begin{lem}\label{intlem2}
Let ${\bf a}:=\{a_1,\ldots,a_A\}$, 
${\bf b}:=\{b_1,\ldots,b_B\}$, 
${\bf c}:=\{c_1,\ldots,c_C\}$, 
${\bf d}:=\{d_1,\ldots,d_D\}$ be
sets of complex numbers with cardinality 
$A, B, C, D\in\mathbb N_0$ (not all zero) respectively, 
$z=\expe^{i\psi}$, $w=\expe^{i\eta}$. 
Let $\sigma, \tau \in \moro{(0,\infty)}$, so that 
$|b_j|<|d_l|<\min\{\moro{\tau},1/\sigma\}$, $|d_la_i|<1$, 
and $|a_i|<|c_k|<\min\{\sigma,1/\moro{\tau}\}$ for any $a_i, b_j, c_k, d_l$ 
elements of ${\bf a}, {\bf b}, {\bf c}, {\bf d}$ respectively. Then 
\begin{equation}\label{invers-rel-gc}
\int_{-\pi}^\pi
\frac{({\bf b}\frac{\sigma}{z}, {\bf a}\frac{z}{\sigma};q)_\infty}
{({\bf d}\frac{\sigma}{z}, {\bf c}\frac{z}{\sigma};q)_\infty}
{\mathrm d}{\psi}
=
\int_{-\pi}^\pi
\frac{({\bf a}\frac{\tau}{w}, {\bf b}\frac{w}{\tau};q)_\infty}
{({\bf c}\frac{\tau}{w}, {\bf d}\frac{w}{\tau};q)_\infty}
{\mathrm d}{\eta}.
\end{equation}
\end{lem}
}


{
The $A=B=C=D=2$ case of {Corollary} \ref{megathm} 
is quite interesting. 
{It is only one example of an infinite sequence of 
such results with arbitrary values of
$A$, $B$, $C$, $D\in\N$ in {Corollary} 
\ref{megathm}---it relates the sum of two ${}_4\phi_3$'s to a different sum of two ${}_4\phi_3$'s and 
provides a generalization 
of Corollary 2.4 in Ismail \& Stanton (2002)
\cite{IsmailStanton2002} (see also Ismail (2005) \cite[Corollary 15.8.3]{Ismail}). 
} 
}

{
\begin{cor}
\label{abcd2}
Let $\moro{t,}a,b,c,d,e,f,g,h\in\CCast$, \moro{$|ab|<|ef|$, $|cd|<|gh|$}. Then
\begin{eqnarray}
&&\hspace{-0.5cm}
\II{e\boro{;}f}
\frac{(etc,etd,\frac{a}{e},\frac{b}{e};q)_\infty}
{(etg,eth,\frac{f}{e};q)_\infty}
\qhyp43{etg,eth,\frac{qe}{a},\frac{qe}{b}}
{etc,etd,\frac{qe}{f}}{q,\frac{ab}{ef}}
\nonumber\\
&&
\hspace{2cm}=\II{g\boro{;}h}
\displaystyle\frac{(gta,gtb,\frac{c}{g},\frac{d}{g};q)_\infty}
{(gte,gtf,\frac{h}{g};q)_\infty}
\qhyp43{gte,gtf,\frac{qg}{c},\frac{qg}{d}}
{gta,gtb,\frac{qg}{h}}{q,\frac{cd}{gh}}.
\end{eqnarray}
\end{cor}
}

The $t=1$ case is interesting.

{
\begin{cor}Let $a,b,c,d,e,f,g,h\in\CCast$, \moro{$|ab|<|ef|$, $|cd|<|gh|$}. Then
\begin{eqnarray}
&&\hspace{-0.5cm}
\II{e\boro{;}f}
\frac{(ec,ed,\frac{a}{e},\frac{b}{e};q)_\infty}{(eg,eh,\frac{f}{e};q)_\infty}
\qhyp43{eg,eh,\frac{qe}{a},\frac{qe}{b}}{ec,ed,\frac{qe}{f}}{q,\frac{ab}{ef}}
\nonumber\\
&&
\hspace{2cm}=\II{g\boro{;}h}
\frac{(ga,gb,\frac{c}{g},\frac{d}{g};q)_\infty}
{(ge,gf,\frac{h}{g};q)_\infty}
\qhyp43{ge,gf,\frac{qg}{c},\frac{qg}{d}}{ga,gb,\frac{qg}{h}}{q,\frac{cd}{gh}}.
\end{eqnarray}
\label{symcor}
\end{cor}
}

{
By exploiting the trick adopted in \cite[Exercise 4.4]{GaspRah}
which converts those basic hypergeometric functions with specific argument to those
with argument $q$ and reduces the number of numerator
parameters and denominator parameters by two.
By mapping
\[
(e,f,a,b,c,d)\mapsto(e,f,\kappa e,qf/\kappa,q/(\kappa e),
\kappa/f),
\]
one converts the left-hand side of Corollary
\ref{symcor} to that with an argument $q$ and reduces the
${}_4\phi_3$'s to ${}_2\phi_1$'s.
}
Furthermore, by mapping $(g,h)\mapsto(1/(\mu e),\mu/f)$,
this converts the right-hand side of the above Corollary
to that with an argument $q$.
{
The resulting relation can be easily verified
using the $q$-Gauss sum \cite[(II.8)]{GaspRah}.}

\section{{Generating functions and integral representations}}
{
One powerful application of integral representations for
basic hypergeometric functions is the determination of 
generating functions for basic hypergeometric orthogonal 
polynomials in the $q$-Askey scheme. 
}
\subsection{{The Askey--Wilson polynomials}}

In this section we study integral representations {for the Askey--Wilson
polynomials and some useful 
applications of these.}

\subsubsection{{Integral representations of the Askey--Wilson polynomials}}

{
A key formula which allows for this is given
in \cite[Exercise 4.5]{GaspRah} {that} is 
equivalent to the following.
}
{
\begin{thm}
Let $a,b,c,d,f\in\CCast$, \moro{$\sigma\in(0,1)$},
$\max(|a|,|b|,|c|,|d|)<1$, 
$q\in\CCdag$,
$x=\cos\theta\in[-1,1]$, $z=\expe^{i\psi}$,
\moro{$f,f{\mathrm e}^{2i\theta}\ne q^m$, $m\in\Z$}. Then
\begin{eqnarray}
&&\hspace{-1.5cm}p_n(x;{\mathfrak a}|q)=
\frac{(q,a\expe^{\pm i\theta},
b\expe^{\pm i\theta},
c\expe^{\pm i\theta};q)_\infty
(ab,ac,bc;q)_n}
{2\pi
{\vartheta(f,f\expe^{2i\theta};q)}
(ab,ac,bc;q)_\infty}
D_n(x;{\mathfrak a},f,\sigma|q),
\end{eqnarray}
where
\begin{eqnarray}
&&\hspace{-1.8cm}D_n(x;{\mathfrak a},f,\sigma|q)=
\int_{-\pi}^\pi\frac{
((f\expe^{i\theta},\frac{q}{f}\expe^{-i\theta})
\frac{\sigma}{z},(f\expe^{i\theta},\frac{q}{f}\expe^{-i\theta},abc)\frac{z}{\sigma};q)_\infty}
{(\expe^{\pm i\theta}\frac{\sigma}{z},(a,b,c)\frac{z}{\sigma};q)_\infty}
\frac{(d\frac{\sigma}{z};q)_n}
{(abc\frac{z}{\sigma};q)_n}
\left(\frac{z}{\sigma}\right)^n{\mathrm d}\psi
\label{IRAW}\\
&&\hspace{0.8cm}
=\int_{-\pi}^\pi\frac{
((fabc\,\expe^{i\theta},\frac{q}{f}abc\,\expe^{-i\theta})\frac{\sigma}{z},(f\frac{1}{abc}\expe^{i\theta},\frac{q}{f}\frac{1}{abc}\expe^{-i\theta},1)\frac{z}{\sigma};q)_\infty}
{(abc\,\expe^{\pm i\theta}\frac{\sigma}{z},(\frac{1}{ab},\frac{1}{ac},\frac{1}{bc})\frac{z}{\sigma};q)_\infty}
\nonumber\\&&\hspace{2.5cm}\times
\frac{(abcd\frac{\sigma}{z};q)_n}
{(\frac{z}{\sigma};q)_n}
\left(\frac{1}{abc}\frac{z}{\sigma}\right)^n{\mathrm d}\psi
.
\label{IRAW2}
\end{eqnarray}
\end{thm}
}
{
\begin{proof}
The integral representation \eqref{IRAW}
is Exercise 4.5 in \cite{GaspRah}. The integral
representation \eqref{IRAW2} is derived as follows. Start with \cite[(30)]{CohlCostasSantosGe} then apply
\cite[(III.23)]{GaspRah}. This produces the 
following nonterminating representation
of the Askey--Wilson polynomials
\begin{eqnarray}
&&\hspace{-1.2cm}\qhyp43{q^{-n},q^{n-1}abcd,a\expe^{\pm i\theta}}{ab,ac,ad}{q,q}=
\frac{(a^2cd,cd;q)_n(\frac{qa}{b},\frac{q}{ab},acd\,\expe^{\pm i\theta};q)_\infty}
{(acd\,\expe^{\pm i\theta};q)_n(\frac{q}{b} \expe^{\pm i\theta},a^2cd,cd;q)_\infty
}
\nonumber\\&&\hspace{3.5cm}\times
\Whyp87{q^{n-1}a^2cd}{q^nac,q^nad,q^{n-1}abcd,a\expe^{\pm i\theta}}{q,\frac{q^{1-n}}{ab}},
\end{eqnarray}
where $|q^{1-n}|<|ab|$. Using this nonterminating
representation, comparing it with 
\eqref{fourfour}, and simplifying 
completes the proof.
\end{proof}
}

\subsubsection{{Generating functions for the Askey--Wilson polynomials}}

{Many researchers have 
investigated series and $q$-integral
{\eqref{qint}}
representations for Askey--Wilson
polynomials. On the other hand,
it seems that regular integral representations
for the Askey--Wilson polynomials have been
largely ignored. In the following 
we will demonstrate how these
representations for the  Askey--Wilson polynomials
allow for simple and straightforward
evaluations of some of their fundamental properties{, particularly their generating functions.}
We begin with the Rahman generating function 
for the Askey--Wilson polynomials, which is the
$q$-analogue of the following generating
function for the Wilson polynomials
\cite[(6.2)]{Ismailetal1990}.
}

For Wilson polynomials there is the following
generating function.
Let $a,b,c,d,t\in\CCast$, 
$\max(|a|,|b|,|c|,|d|,|t|)<1$,
$x=\cos\theta\boro{\in[-1,1]}$,
\moro{$4|t|<|1-t|^2$,}
\begin{eqnarray}
&&\hspace{-1cm}\sum_{n=0}^\infty
\frac{(a+b+c+d-1)_n}
{n!(a+b,a+c,a+d)_n}W_n(x^2;{\mathfrak a})t^n\nonumber\\
&&=(1-t)^{1-a-b-c-d}\hyp43
{\frac12(a+b+c+d-1),\frac12(a+b+c+d),
a\pm ix}{a+b,a+c,a+d}{\frac{-4t}{(1-t)^2}}.
\label{Wilsgf}
\end{eqnarray}


\noindent
Rahman computed a $q$-analogue
of \eqref{Wilsgf} in 
\cite[(4.9)]{Rahman96}
{by using}
a $q$-integral representation of Askey--Wilson
polynomials. We will prove the 
same generating function using
the above integral representation
\eqref{IRAW}.
\begin{thm}[Rahman (1996)]
Let $k,p\in\{1,2,3,4\}$,
${\bf a}:=\moro{\{}a_1,a_2,a_3,a_4\moro{\}}$,
$t,a_k\in\moro{\mathbb C^\ast}$,
$x=\cos\theta\in[-1,1]$,
$q\in\CCdag$,
\boro{$|a_pt|<1$}. Then
\begin{eqnarray}
&&\hspace{-0.5cm}\sum_{n=0}^\infty \frac{t^n\,(q^{-1}a_{1234};q)_n\,
p_n(x;{\bf a}|q)}
{(q,\{a_pa_s\}_{s\ne p};q)_n}\nonumber\\
&&\hspace{-0.0cm}=\frac{(ta_{1234}(qa_p)^{-1};q)_\infty}{(ta_p^{-1};q)_\infty}\,
\qhyp65
{ \pm (q^{-1}a_{1234})^\frac12, \pm (a_{1234})^\frac12,
a_p \boro{\expe}^{\pm i\theta}}
{\{a_pa_s\}_{s\ne p},
ta_{1234}(qa_p)^{-1},qa_pt^{-1}}
{q,q}
\nonumber\\
&&\hspace{0.2cm}+\frac{(
\{ta_s\}_{s\ne p},
q^{-1}a_{1234},
a_p\boro{\expe}^{\pm i\theta};q)_\infty}
{(\{a_pa_s\}_{s\ne p},a_pt^{-1},t\boro{\expe}^{\pm i\theta};q)_\infty}
\qhyp65{\pm ta_p^{-1}(q^{-1}a_{1234})^\frac12,\pm
ta_p^{-1}(a_{1234})^\frac12,
t \boro{\expe}^{\pm i\theta}}
{\{ta_s\}_{s\ne p},\boro{q^{-1} a_{1234}(t a_p^{-1})^2},
qta_p^{-1}}{q,q}.
\label{genfun2ask}
\end{eqnarray}
\end{thm}
\begin{proof}
Start with the left-hand side of 
\eqref{genfun2ask} and insert
\eqref{IRAW}. This produces 
\begin{eqnarray}
&&\sum_{n=0}^\infty
\frac{(abcd/q;q)_nt^n\,p_n(x;{\mathfrak a}|q)}{(q,ab,ac,ad;q)_n}=
\frac{(q,a\expe^{\pm i\theta},
b\expe^{\pm i\theta},
c\expe^{\pm i\theta};q)_\infty
}
{2\pi
{\vartheta(f,f\expe^{2i\theta};q)}
(ab,ac,bc;q)_\infty}
\nonumber\\
&&\hspace{1cm}
\times\int_{-\pi}^\pi\frac{
((f\expe^{i\theta},\frac{q}{f}\expe^{-i\theta})
\frac{\sigma}{z},(f\expe^{i\theta},\frac{q}{f}\expe^{-i\theta},abc)\frac{z}{\sigma};q)_\infty}
{(\expe^{\pm i\theta}\frac{\sigma}{z},(a,b,c)\frac{z}{\sigma};q)_\infty}
\qhyp32{d\frac{\sigma}{z},abcd/q,bc}
{abc\frac{z}{\sigma},ad}{q,\frac{tz}{\sigma}}
{\mathrm d}\psi.
\end{eqnarray}
The ${}_3\phi_2$ can be written as a sum
of two ${}_5\phi_4(q,q)$ using 
\cite[(3.4.1)]{GaspRah}, since it is 
well-poised. Comparing 
this sum using 
\eqref{qqform} (see also Corollary \ref{3phi2thm}) produces
\begin{eqnarray}
&&\hspace{-0.5cm}\qhyp32{d\frac{\sigma}{z},abcd/q,bc}
{abc\frac{z}{\sigma},ad}{\boro{q,} \frac{tz}{\sigma}}
=\frac{(q,abcd/q,\frac{az}{\sigma};q)_\infty}{2\pi
{\vartheta(h,h\frac{a}{t};q)}
(abc\frac{z}{\sigma},ad;q)_\infty}\nonumber\\
&&\hspace{0.5cm}\times\int_{-\pi}^\pi
\frac{((h\sqrt{\frac{a}{t}},\frac{q}{h}\sqrt{\frac{t}{a}})\frac{\tau}{w},(h\sqrt{\frac{a}{t}},\frac{q}{h}\sqrt{\frac{t}{a}},d\sqrt{ta},\frac{bcdt^{3/2}}{q\sqrt{a}},bc\sqrt{ta}\frac{z}{\sigma})\frac{w}{\tau};q)_\infty}
{((\sqrt{\frac{a}{t}},\sqrt{\frac{t}{a}})\frac{\tau}{w},(\pm\sqrt{\frac{tbcd}{q}},\pm\sqrt{tbcd},\sqrt{ta}\frac{z}{\sigma})\frac{w}{\tau};q)_\infty}{\mathrm d}{\eta},
\end{eqnarray}
where $w=\expe^{i\eta}$. Inserting the above integral
representation, rearranging
the integrals and evaluating
the outer integral using Gasper's
integral \eqref{Gint} completes the proof.
\end{proof}

We can also derive an integral representation for a product
of two ${}_2\phi_1$'s by using the other generating
function which is known for Askey--Wilson polynomials
\cite[(14.1.15)]{Koekoeketal}. Integral representations
for products of basic hypergeometric functions is an
interesting direction of research.

\begin{thm}
Let $a,b,c,d,t,f\in\CCast$, $q\in\CCdag$, \moro{$\sigma\in(0,1)$,} $|t|<\moro{\sigma}$,
\moro{$f,f\expe^{2i\theta}\ne q^m$, $m\in\Z$},
$z=\expe^{i\psi}$. Then
\begin{eqnarray}
&&\hspace{-0.5cm}
\int_{-\pi}^\pi\frac{
((f\expe^{i\theta},\frac{q}{f}\expe^{-i\theta})
\frac{\sigma}{z},(f\expe^{i\theta},\frac{q}{f}\expe^{-i\theta},abc)
\frac{z}{\sigma};q)_\infty}
{(\expe^{\pm i\theta}\frac{\sigma}{z},(a,b,c)\frac{z}{\sigma};q)_\infty}
\qhyp32{d\frac{\sigma}{z},ab,ac}
{abc\frac{z}{\sigma},ad}{q,\frac{tz}{\sigma}}
{\mathrm d}\psi\nonumber\\
&&\hspace{-0.2cm}=
2\pi\frac{
{\vartheta(f,f\expe^{2i\theta};q)}
(ab,ac,bc;q)_\infty}{(q,a\expe^{\pm i\theta},
b\expe^{\pm i\theta},c\expe^{\pm i\theta};q)_\infty
}
\!\qhyp21{a\expe^{i\theta},d\expe^{i\theta}}{ad}{q,t\expe^{-i\theta}}
\!\qhyp21{b\expe^{-i\theta},c\expe^{-i\theta}}{bc}{q,t\expe^{i\theta}}\!.
\end{eqnarray}
\end{thm}
\begin{proof}
Start with the generating function for the Askey--Wilson polynomials 
\cite[(1.9)]{IsmailWilson82},
\cite[(14.1.15)]{Koekoeketal},
\begin{eqnarray}
\sum_{n=0}^\infty\frac{t^n\,p_n(x;{\mathfrak a}|q)}
{(q,ad,bc;q)_n}=
\qhyp21{a\expe^{i\theta},d\expe^{i\theta}}{ad}{q,t\expe^{-i\theta}}
\qhyp21{b\expe^{-i\theta},c\expe^{-i\theta}}{bc}{q,t\expe^{i\theta}}.
\label{AWgf2}
\end{eqnarray}
Inserting the integral representation 
\eqref{IRAW} into the left-hand side of
\eqref{AWgf2} produces the
following integral representation
for \eqref{AWgf2}, namely
\begin{eqnarray}
&&\sum_{n=0}^\infty
\frac{t^n\,p_n(x;{\mathfrak a}|q)}{(q,ad,bc;q)_n}=
\frac{(q,a\expe^{\pm i\theta},
b\expe^{\pm i\theta},c\expe^{\pm i\theta};q)_\infty
}{2\pi
{\vartheta(f,f\expe^{2i\theta};q)}
(ab,ac,bc;q)_\infty}
\nonumber\\
&&\hspace{1cm}
\times\int_{-\pi}^\pi\frac{
((f\expe^{i\theta},\frac{q}{f}\expe^{-i\theta})
\frac{\sigma}{z},(f\expe^{i\theta},\frac{q}{f}\expe^{-i\theta},abc)
\frac{z}{\sigma};q)_\infty}
{(\expe^{\pm i\theta}\frac{\sigma}{z},(a,b,c)\frac{z}
{\sigma};q)_\infty}
\qhyp32{d\frac{\sigma}{z},ab,ac}
{abc\frac{z}{\sigma},ad}{q,\frac{tz}{\sigma}}
{\mathrm d}\psi.
\label{intrefAW2}
\end{eqnarray}
Comparing \eqref{AWgf2} with  \eqref{intrefAW2}
completes the proof.
\end{proof}
\subsection{Continuous dual $q$-Hahn polynomials}

If you let $a_4\to 0$ \moro{$(d\to 0)$} in the Askey--Wilson polynomials you obtain
the three parameter symmetric 
continuous dual $q$-Hahn polynomials
\cite[Section 14.3]{Koekoeketal}.
In this case the Askey--Wilson polynomials
$p_n(x;{\bf a}|q)$ with ${\bf a}=\{a_1,a_2,a_3,a_4\}$ reduce to
the continuous dual $q$-Hahn polynomials
$p_n(x;{\bf a}|q)$ with ${\bf a}:=\{a_1,a_2,a_3\}$,
\moro{${\mathfrak a}:=\{a,b,c\}$, ${\mathfrak a}={\mathbf a}$}.

\subsubsection{Integral representations for the continuous dual $q$-Hahn polynomials}

One can obtain several integral representations for 
the continuous dual $q$-Hahn polynomials by
starting with \eqref{IRAW}.
\begin{cor}
Let $a,b,c,f\in\CCast$, 
$\max(|a|,|b|,|c|)<1$, \moro{$\sigma\in(0,1)$,} $q\in\CCdag$,
$x=\cos\theta$\boro{$\in[-1,1]$}, $z=\expe^{i\psi}$,
\moro{$f,f\expe^{2i\theta}\ne q^m$, $m\in\Z$}. Then
\begin{eqnarray}
&&\hspace{-1.5cm}p_n(x;{\mathfrak a}|q)=
\frac{(q,a\expe^{\pm i\theta},
b\expe^{\pm i\theta};q)_\infty
(ab;q)_n}
{2\pi
{\vartheta(f,f\expe^{2i\theta};q)}
(ab;q)_\infty}E_n({\mathfrak a};f,\sigma|q)\\[0.2cm]
&&\hspace{0.2cm}=
\frac{(q,a\expe^{\pm i\theta},
b\expe^{\pm i\theta},c\expe^{\pm i\theta};q)_\infty
(ab,ac,bc;q)_n}
{2\pi
{\vartheta(f,f\expe^{2i\theta};q)}
(ab,ac,bc;q)_\infty}
F_n({\mathfrak a};f,\sigma|q),
\end{eqnarray}
where
\begin{eqnarray}
&&\hspace{-1.7cm}E_n({\mathfrak a};f,\sigma|q)
\label{IRcdqE}
=\int_{-\pi}^\pi\frac{
((f\expe^{i\theta},\frac{q}{f}\expe^{-i\theta})
\frac{\sigma}{z},(f\expe^{i\theta},\frac{q}{f}
\expe^{-i\theta},abc)\frac{z}{\sigma};q)_\infty}
{(\expe^{\pm i\theta}\frac{\sigma}{z},(a,b)\frac{z}
{\sigma};q)_\infty}
\left(c\text{\scriptsize{$\frac{\sigma}{z}$}};q\right)_n
\left(\frac{z}{\sigma}\right)^n{\mathrm d}\psi,
\end{eqnarray}
\begin{eqnarray}
&&\hspace{-0.5cm}
F_n({\mathfrak a};f,\sigma|q)=
\int_{-\pi}^\pi\frac{
((f\expe^{i\theta},\frac{q}{f}\expe^{-i\theta})
\frac{\sigma}{z},(f\expe^{i\theta},\frac{q}
{f}\expe^{-i\theta},abc)\frac{z}{\sigma};q)_\infty}
{(\expe^{\pm i\theta}\frac{\sigma}{z},(a,b,c)\frac{z}
{\sigma};q)_\infty(abc\frac{z}{\sigma};q)_n}
\left(\frac{z}{\sigma}\right)^n{\mathrm d}\psi
\label{IRcdqH}\\[0.2cm]
&&\hspace{1.7cm}=\int_{-\pi}^\pi\frac{
((fabc\,\expe^{i\theta},\frac{q}{f}abc\,\expe^{-i\theta})
\frac{\sigma}{z},
(f\frac{1}{abc}\expe^{i\theta},\frac{q}{f}\frac{1}{abc}
\expe^{-i\theta},1)\frac{z}{\sigma};q)_\infty}
{(abc\,\expe^{\pm i\theta}\frac{\sigma}{z},(\frac{1}
{ab},\frac{1}{ac},\frac{1}{bc})\frac{z}
{\sigma};q)_\infty(\frac{z}{\sigma};q)_n}
\left(\frac{1}{abc}\frac{z}{\sigma}\right)^n{\mathrm d}\psi.
\label{IRcdqH2}
\end{eqnarray}
\end{cor}
\begin{proof}
Starting with \eqref{IRAW}, letting ${d}\mapsto 0$ 
produces \eqref{IRcdqH}, and taking
${c}\mapsto 0$ followed by
${d}\mapsto {c}$ produces \eqref{IRcdqE}.
Starting with \eqref{IRAW2}, letting
${d}\mapsto 0$ produces 
\eqref{IRcdqH2}.  This completes
the proof.
\end{proof}
\begin{rem}
The continuous dual $q$-Hahn polynomials
are symmetric in the three parameters $a$, $b$, and $c$. 
The symmetry in the parameters is evident in the 
integral representation
\eqref{IRcdqH2}.
\end{rem}

Note that starting with
\eqref{IRAW2} and taking either $a$, $b$ or $c$ 
$\mapsto 0$ does not yield a finite result.

\subsubsection{Generating functions for the continuous dual $q$-Hahn polynomials}
\medskip
There are several generating functions known
for the continuous dual $q$-Hahn polynomials.
Some of them follow by taking the $a_4\to 0$ limit for
generating functions of the Askey--Wilson 
polynomials. One such example which hasn't appeared
frequently in the literature is the $a_4\to 0$ limit
of the Rahman generating function \eqref{genfun2ask}.
This is given as follows.

\begin{cor}
Let $k,p\in\{1,2,3\}$,
${\bf a}:=\moro{\{}a_1,a_2,a_3\moro{\}}$,
$a_k\moro{,t}\in\moro{{\mathbb C}^\ast}$,
$x=\cos\theta\in[-1,1]$,
\boro{$q\in \CCdag$},
$|t|<1$. Then
\begin{eqnarray}
&&\hspace{-1.5cm}\sum_{n=0}^\infty 
\frac{p_n(x;{\bf a}|q)}{(q,\{a_pa_s\}_{s\ne p};q)_n}t^n
=\frac{1}{(ta_p^{-1};q)_\infty}\,\qhyp43
{a_p \boro{\expe}^{\pm i\theta},0,0}
{\{a_pa_s\}_{s\ne p}
,qa_pt^{-1}}
{q,q}
\nonumber\\
&&\hspace{2.2cm}+\frac{(
\{ta_s\}_{s\ne p},
a_p\boro{\expe}^{\pm i\theta};q)_\infty}
{(\{a_pa_s\}_{s\ne p},a_pt^{-1},t\boro{\expe}^{\pm i\theta};q)_\infty}
\qhyp43{t \boro{\expe}^{\pm i\theta},0,0}
{\{ta_s\}_{s\ne p},
qta_p^{-1}}{q,q}
.
\label{genfun2cdH}
\end{eqnarray}
\end{cor}

\medskip
A non-standard generating function for 
continuous dual $q$-Hahn polynomials was presented in
\cite[(3.5)]{Atak2011}. We will show how integral representations
for continuous dual $q$-Hahn polynomials {lead} to an
easy proof of this formula.
\begin{thm}[Atakishiyeva \& Atakishiyev (2011)]
Let $a,b,c\in\CCast$, $q\in\CCdag$, 
$t\in\CC$, such that $|t|<1$. Then
\begin{eqnarray}
&&\sum_{n=0}^\infty\frac{t^n\,p_n(x;{\mathfrak a}|q)}
{(q,tabc;q)_n}
=\frac{(ta,tb,tc;q)_\infty}{(tabc,t\expe^{\pm i\theta};q)_\infty}.
\label{nonstan}
\end{eqnarray}
\end{thm}
\begin{proof}
{Starting} with the left-hand side of \eqref{nonstan} and {inserting} \eqref{IRcdqH}, one obtains
\begin{eqnarray}
&&\sum_{n=0}^\infty\frac{t^n\,p_n(x;{\mathfrak a}|q)}{(q,tabc;q)_n}=
\frac{(q,a\expe^{\pm i\theta},b\expe^{\pm i\theta},c\expe^{\pm i\theta};q)_\infty}
{
{\vartheta(f,f\expe^{2i\theta};q)}
(ab,ac,bc;q)_\infty}\nonumber\\
&&\hspace{1.5cm}\times\int_{-\pi}^\pi
\frac{((f\expe^{i\theta},\frac{q}{f}\expe^{-i\theta})\frac{\sigma}{z},(f\expe^{i\theta},\frac{q}{f}\expe^{-i\theta},abc)\frac{z}{\sigma};q)_\infty}
{(\expe^{\pm i\theta}\frac{\sigma}{z},(a,b,c)\frac{z}{\sigma};q)_\infty}
\qhyp32{ab,ac,bc}{abc\frac{z}{\sigma},tabc}{q,\frac{tz}{\sigma}}
{\mathrm d}\psi,
\label{prenon}
\end{eqnarray}
where $z=\expe^{i\psi}$. Using \cite[(III.34)]{GaspRah}, we can write 
the ${}_3\phi_2$ in the integrand as a sum of two ${}_3\phi_2$'s
with argument $q$. Then using \eqref{qqform} we can express it as
an integral representation, namely
\begin{eqnarray}
&&\hspace{-0.5cm}\qhyp32{ab,ac,bc}{abc\frac{z}{\sigma},tabc}{q,\frac{tz}{\sigma}}=\frac{(q,tb,tc,ab,ac,\frac{az}{\sigma};q)_\infty}
{2\pi
{\vartheta(f,f\frac{a}{t};q)}
(tabc,
abc\frac{z}{\sigma};q)_\infty}\nonumber\\
&&\hspace{3cm}\times
\int_{-\pi}^{\pi}\frac{
((f\sqrt{\frac{a}{t}},\frac{q}{f}\sqrt{\frac{t}{a}})\frac{\tau}{w},(f\sqrt{\frac{a}{t}},\frac{q}{f}\sqrt{\frac{t}{a}},
\sqrt{ta}\,bc\frac{z}{\sigma})\frac{w}{\tau};q)_\infty}
{((\sqrt{\frac{a}{t}},\sqrt{\frac{t}{a}})\frac{\tau}{w},(\sqrt{ta}\,b,\sqrt{ta}\,c,\sqrt{ta}\,\frac{z}{\sigma})\frac{w}{\tau};q)_\infty}
{\mathrm d}\eta,
\label{nonir}
\end{eqnarray}
where $w=\expe^{i\eta}$\moro{,}
\moro{and none of the arguments of the modified theta functions are equal to some $q^m$, $m\in{\mathbb Z}$.}
Inserting the integral representation
\eqref{nonir} into the right-hand side of \eqref{prenon}
and using Gasper's integral \eqref{Gint} completes the proof.
\end{proof}

\medskip
%
For continuous dual Hahn polynomials
\cite[Section 9.3]{Koekoeketal}
$S_n(x^2;{\bf a})$,
there is a generating function
which until recently there has been no known $q$-analogue for.
This is the generating function 
\cite[(9.3.16)]{Koekoeketal}
\begin{eqnarray}
&&\sum_{n=0}^\infty
\frac{(\gamma)_n
S_n(x^2;{\mathfrak a})}
{n!(a+b,a+c)_n}t^n
=(1-t)^{-\gamma}
\hyp32{\gamma,a\pm ix}
{a+b,a+c}{\frac{t}{t-1}},
\end{eqnarray}
where $\gamma$ is a free parameter, $|t| < 1$, $|t| < |1- t|$.
{Using the integral representation method
we may readily compute the following $q$-analogue
which we present now.}

{
\begin{thm}
Let $q \in \CCdag$, $\gamma\in \mathbb C$, $t, a, b, c\in \mathbb C^\ast$, $|t| < 1$. 
Let $\gamma\in\mathbb C$. Then
one has the following 
generating function for
continuous dual $q$-Hahn 
polynomials
\begin{eqnarray}
&&\hspace{-0.6cm}\sum_{n=0}^\infty \frac{(\gamma;q)_n
p_n(x;{\mathfrak a}|q)}
{(q,ab,ac;q)_n}t^n
=\frac{(a\expe^{\pm i\theta};q)_\infty}
{(ab,ac;q)_\infty}
\nonumber\\
&&\hspace{-0.4cm}\times
\left(\frac{(ab,ac,\gamma t/a;q)_\infty}
{(a\expe^{\pm i\theta},t/a;q)_\infty}\qhyp43{\gamma,a\expe^{\pm i\theta},0}
{ab,ac,qa/t}{q,q}+
\frac{(tb,tc,\gamma;q)_\infty}{(t\expe^{\pm i\theta},a/t;q)_\infty}
\qhyp43{\gamma t/a,t\expe^{\pm i\theta},0}
{tb,tc,qt/a}{q,q}\right).
\label{missing}
\end{eqnarray}
\end{thm}
}
\begin{proof}
Start with the left-hand side of 
\eqref{missing} and insert
\eqref{IRcdqH2}. This produces 
\begin{eqnarray}
&&\sum_{n=0}^\infty
\frac{(\gamma;q)_nt^n\,p_n(x;{\mathfrak a}|q)}{(q,ab,ac;q)_n}=
\frac{(q,a\expe^{\pm i\theta},
b\expe^{\pm i\theta},
c\expe^{\pm i\theta};q)_\infty
}
{2\pi
{\vartheta(f,f\expe^{2i\theta};q)}
(ab,ac,bc;q)_\infty}
\nonumber\\
&&\hspace{1cm}
\times\int_{-\pi}^\pi\frac{
((fabc\,\expe^{i\theta},\frac{q}{f}abc\,\expe^{-i\theta})
\frac{\sigma}{z},
(f\frac{1}{abc}\expe^{i\theta},\frac{q}{f}
\frac{1}{abc}\expe^{-i\theta},1)\frac{z}{\sigma};q)_\infty}
{(abc\,\expe^{\pm i\theta}\frac{\sigma}{z},(\frac{1}{ab},\frac{1}{ac},\frac{1}{bc})\frac{z}{\sigma};q)_\infty}
\qhyp21{\gamma,bc}
{\frac{z}{\sigma}}{q,\frac{tz}{abc\sigma}}
{\mathrm d}\psi.
\end{eqnarray}
The ${}_2\phi_1$ can be written either
as a sum of two 
\moro{$\nqphyp{2}{2}{-1}$}'s 
\cite[(17.9.3)]{NIST:DLMF}
or as a sum of two 
\moro{$\nqphyp{3}{1}{1}$}'s
\cite[(17.9.3\_5)]{NIST:DLMF}.
We use \eqref{2phi1a} which corresponds to the expansion
of a ${}_2\phi_1$ with a ${}_3\phi_2$ with one vanishing
numerator parameter.
Comparing 
this sum using 
\eqref{qqform} produces
\begin{eqnarray}
&&\hspace{-1.2cm}\qhyp21{\gamma,bc}
{\frac{z}{\sigma}}{q,\frac{tz}{abc\sigma}}
=\frac{(q,\gamma,
\gamma \frac{t}{a},\frac{z}{bc\sigma};q)_\infty}
{2\pi
{\vartheta(h,h\frac{a}{t};q)}
(\frac{z}{\sigma};q)_\infty}
\nonumber\\&&\hspace{3.5cm}\times
\int_{-\pi}^\pi\frac{((h\sqrt{\frac{a}{t}},\frac{q}{h}
\sqrt{\frac{t}{a}})\frac{\tau}{w},(h\sqrt{\frac{a}{t}},\frac{q}{h}
\sqrt{\frac{t}{a}},\sqrt{\frac{t}{a}}\frac{z}{\sigma})
\frac{w}{\tau};q)_\infty}{((\sqrt{\frac{a}{t}},\sqrt{\frac{t}{a}})
\frac{\tau}{w},(\gamma\sqrt{\frac{t}{a}},\sqrt{\frac{t}{a}}
\frac{z}{bc\sigma})\frac{w}{\tau};q)_\infty}{\mathrm d}{\eta},
\end{eqnarray}
where $w=\expe^{i\eta}$\moro{,}
\moro{and none of the arguments of the modified theta functions are equal to some $q^m$, $m\in{\mathbb Z}$.}
Inserting the above integral
representation, rearranging
the integrals and evaluating
the outer integral using Gasper's
integral \eqref{Gint} completes the proof.
\end{proof}

{
If one {lets} $\gamma\to 0$ in \eqref{missing}
then one obtains \eqref{genfun2cdH}.
Furthermore, 
if you then let $a_3\to 0$ you produce the
following well-known generating
function for Al-Salam--Chihara polynomials
\cite[(14.8.16)]{Koekoeketal}
\begin{eqnarray}
&&\hspace{-0.6cm}\sum_{n=0}^\infty \frac{(\gamma;q)_nt^n
p_n(x;{\mathfrak a}|q)}
{(q,ab;q)_n}
=\frac{(a\expe^{\pm i\theta};q)_\infty}
{(ab;q)_\infty}
\nonumber\\
&&\hspace{-0.0cm}\times
\left(\frac{(ab,\gamma t/a;q)_\infty}{(a\expe^{\pm i\theta},t/a;q)_\infty}
\qhyp32{\gamma,a\expe^{\pm i\theta}}
{ab,qa/t}{q,q}+
\frac{(tb,\gamma;q)_\infty}{(t\expe^{\pm i\theta},a/t;q)_\infty}
\qhyp32{\gamma t/a,t\expe^{\pm i\theta}}
{tb,qt/a}{q,q}\right)\nonumber\\
&&=\frac{(\gamma t\expe^{i\theta};q)_\infty}{(t\expe^{i\theta};q)_\infty}
\qhyp32{\gamma,a\expe^{i\theta},b\expe^{i\theta}}
{ab,\gamma t\expe^{i\theta}}{q,t\expe^{-i\theta}},
\end{eqnarray}
where the second equality 
is obtained by using the nonterminating
basic hypergeometric series transformation
\cite[(III.34)]{GaspRah}.}


\section*{Acknowledgements}
Much appreciation to Tom Koornwinder for inviting 
the authors to investigate 
integral representations for basic
hypergeometric orthogonal polynomials.
We would also like to especially thank
George Gasper for sharing his deep knowledge on the
subject contained in this manuscript,
for many helpful remarks he had made and for the
many useful conversations we have had.
The first author would like to thank Dennis Stanton for 
pointing out the non-standard generating
function for continuous dual $q$-Hahn polynomials
\eqref{nonstan}.
{We would also like to thank
the referee for many helpful suggestions and especially for their
insight into Dick Askey's viewpoint.}
R.S.C-S acknowledges financial  support  through  the  research  project  PGC2018–096504-B-C33  supported  by  Agencia  Estatal  de Investigaci\'on of Spain.

\bibliographystyle{plain}
\bibliography{refbib} 

\def\cprime{$'$} \def\dbar{\leavevmode\hbox to 0pt{\hskip.2ex \accent"16\hss}d}
\begin{thebibliography}{10}

\bibitem{AAR}
G.~E. Andrews, R.~Askey, and R.~Roy.
\newblock {\em Special functions}, volume~71 of {\em Encyclopedia of
  Mathematics and its Applications}.
\newblock Cambridge University Press, Cambridge, 1999.

\bibitem{AskeyRoy86}
R.~Askey and R.~Roy.
\newblock More {$q$}-beta integrals.
\newblock {\em The Rocky Mountain Journal of Mathematics}, 16(2):365--372,
  1986.

\bibitem{Atak2011}
M.~Atakishiyeva and N.~Atakishiyev.
\newblock A non-standard generating function for continuous dual $q$-{H}ahn
  polynomials.
\newblock {\em Revista de Matem\'{a}tica: Teor\'{i}a y Aplicaciones},
  18(1):111--120, 2011.

\bibitem{Bailey64}
W.~N. {Bailey}.
\newblock {\em Generalized hypergeometric series}.
\newblock Cambridge Tracts in Mathematics and Mathematical Physics, No. 32.
  Stechert-Hafner, Inc., New York, 1964.

\bibitem{CohlCostasSantosGe}
H.~S. {Cohl}, R.~S. {Costas-Santos}, and L.~Ge.
\newblock {Terminating Basic Hypergeometric Representations and Transformations
  for the Askey-Wilson Polynomials}.
\newblock {\em Symmetry}, 12(8):14, 2020.

\bibitem{Gasper1989}
G.~Gasper.
\newblock {$q$}-extensions of {B}arnes', {C}auchy's, and {E}uler's beta
  integrals.
\newblock In {\em Topics in mathematical analysis}, volume~11 of {\em Ser. Pure
  Math.}, pages 294--314. World Sci. Publ., Teaneck, NJ, 1989.

\bibitem{GaspRah}
G.~Gasper and M.~Rahman.
\newblock {\em Basic hypergeometric series}, volume~96 of {\em Encyclopedia of
  Mathematics and its Applications}.
\newblock Cambridge University Press, Cambridge, second edition, 2004.
\newblock With a foreword by Richard Askey.

\bibitem{Ismail}
M.~E.~H. Ismail.
\newblock {\em Classical and {Q}uantum {O}rthogonal {P}olynomials in {O}ne
  {V}ariable}, volume~98 of {\em Encyclopedia of Mathematics and its
  Applications}.
\newblock Cambridge University Press, Cambridge, 2005.
\newblock With two chapters by Walter Van Assche.

\bibitem{Ismailetal1990}
M.~E.~H. Ismail, J.~Letessier, G.~Valent, and J.~Wimp.
\newblock Two families of associated {W}ilson polynomials.
\newblock {\em Canadian Journal of Mathematics. Journal Canadien de
  Math\'{e}matiques}, 42(4):659--695, 1990.

\bibitem{IsmailStanton2002}
M.~E.~H. Ismail and D.~Stanton.
\newblock {$q$}-integral and moment representations for {$q$}-orthogonal
  polynomials.
\newblock {\em Canadian Journal of Mathematics. Journal Canadien de
  Math\'{e}matiques}, 54(4):709--735, 2002.

\bibitem{IsmailStanton15}
M.~E.~H. Ismail and D.~Stanton.
\newblock Expansions in the {A}skey-{W}ilson polynomials.
\newblock {\em Journal of Mathematical Analysis and Applications},
  424(1):664--674, 2015.

\bibitem{IsmailWilson82}
M.~E.~H. Ismail and J.~A. Wilson.
\newblock Asymptotic and generating relations for the {$q$}-{J}acobi and
  {$_{4}\varphi _{3}$} polynomials.
\newblock {\em Journal of Approximation Theory}, 36(1):43--54, 1982.

\bibitem{Koekoeketal}
R.~Koekoek, P.~A. Lesky, and R.~F. Swarttouw.
\newblock {\em Hypergeometric orthogonal polynomials and their
  {$q$}-analogues}.
\newblock Springer Monographs in Mathematics. Springer-Verlag, Berlin, 2010.
\newblock With a foreword by Tom H. Koornwinder.

\bibitem{NassrallahRahman85}
B.~Nassrallah and M.~Rahman.
\newblock Projection formulas, a reproducing kernel and a generating function
  for {$q$}-{W}ilson polynomials.
\newblock {\em SIAM Journal on Mathematical Analysis}, 16(1):186--197, 1985.

\bibitem{NIST:DLMF}
{\it NIST Digital Library of Mathematical Functions}.
\newblock \href{https://dlmf.nist.gov/}{\bf\tt\normalsize
  https://dlmf.nist.gov/}, Release 1.1.5 of 2022-03-15.
\newblock F.~W.~J. Olver, A.~B. {Olde Daalhuis}, D.~W. Lozier, B.~I. Schneider,
  R.~F. Boisvert, C.~W. Clark, B.~R. Miller, B.~V. Saunders, H.~S. Cohl, and
  M.~A. McClain, eds.

\bibitem{Rahman10W986}
M.~Rahman.
\newblock An integral representation of a {$_{10}\varphi_9$} and continuous
  bi-orthogonal {$_{10}\varphi_9$} rational functions.
\newblock {\em Canad. J. Math.}, 38(3):605--618, 1986.

\bibitem{Rahman86prodctsqJ}
M.~Rahman.
\newblock A product formula for the continuous {$q$}-{J}acobi polynomials.
\newblock {\em Journal of Mathematical Analysis and Applications},
  118(2):309--322, 1986.

\bibitem{Rahman88}
M.~Rahman.
\newblock Some extensions of {A}skey-{W}ilson's {$q$}-beta integral and the
  corresponding orthogonal systems.
\newblock {\em Canadian Mathematical Bulletin. Bulletin Canadien de
  Math\'{e}matiques}, 31(4):467--476, 1988.

\bibitem{Rahman96}
M.~Rahman.
\newblock Some generating functions for the associated {A}skey-{W}ilson
  polynomials.
\newblock {\em Journal of Computational and Applied Mathematics},
  68(1-2):287--296, 1996.

\bibitem{Slater66}
L.~J. Slater.
\newblock {\em Generalized hypergeometric functions}.
\newblock Cambridge University Press, Cambridge, 1966.

\bibitem{vandeBultRains09}
F.~J. van~de Bult and E.~M. Rains.
\newblock Basic hypergeometric functions as limits of elliptic hypergeometric
  functions.
\newblock {\em Symmetry, Integrability and Geometry: Methods and Applications},
  5(059), 2009.

\end{thebibliography}

\end{document}